\documentclass[12pt,a4paper]{amsart}

\usepackage[margin=1in]{geometry}
\allowdisplaybreaks[1]

\usepackage{amsmath}
\usepackage{amsfonts}
\usepackage{graphicx}
\usepackage{framed}
\usepackage{amssymb}
\usepackage{esvect}
\usepackage{xcolor}
\usepackage{cite}

\usepackage[backref=page]{hyperref}

\renewcommand*{\backref}[1]{}
\renewcommand*{\backrefalt}[4]{%
    \ifcase #1 (Not cited.)%
    \or        (Cited on page~#2.)%
    \else      (Cited on pages~#2.)%
    \fi}

\usepackage{etex}
\usepackage{latexsym}
\usepackage[english]{babel}
\usepackage[utf8x]{inputenc}

\def \N {\mathbb{N}}
\def \R {\mathbb{R}}

\renewcommand{\epsilon}{\varepsilon}

\usepackage{amsthm}
\theoremstyle{definition}
\newtheorem{definition}{Definition}[section]

\theoremstyle{plain}
\newtheorem{theorem}[definition]{Theorem}
\newtheorem{proposition}[definition]{Proposition}
\newtheorem{lemma}[definition]{Lemma}
\newtheorem{corollary}[definition]{Corollary}

\renewcommand{\ge}{\geqslant}
\renewcommand{\le}{\leqslant}

\usepackage{mathtools}
\mathtoolsset{showonlyrefs}

\numberwithin{equation}{section}

\usepackage[alphabetic]{amsrefs}

\begin{document}

\begin{abstract}
We prove that the trace of nonlocal minimal graphs
at points of stickiness is of class~$C^{1,\gamma}$.

As a result, we show that boundary continuity implies
boundary differentiability for nonlocal minimal graphs.
\end{abstract}
 
\title[Regularity of the trace of nonlocal minimal graphs]{Regularity of the trace of nonlocal minimal graphs}

\author[S. Dipierro]{Serena Dipierro}
\address{S. D., 
Department of Mathematics and Statistics,
University of Western Australia,
35~Stirling Highway, WA 6009 Crawley, Australia. }

\email{serena.dipierro@uwa.edu.au}

\author[O. Savin]{Ovidiu Savin}
\address{O. S., Department of Mathematics,
Columbia University,
2990 Broadway, NY 10027 New York, US. }

\email{savin@math.columbia.edu}

\author[E. Valdinoci]{Enrico Valdinoci}
\address{E. V., 
Department of Mathematics and Statistics,
University of Western Australia,
35~Stirling Highway, WA 6009 Crawley, Australia. }

\email{enrico.valdinoci@uwa.edu.au}

\thanks{SD is supported by the
Australian Future Fellowship FT230100333. OS is supported by the NSF Grant DMS-2349794.
EV is supported by the Australian Laureate Fellowship FL190100081.}



\maketitle

\tableofcontents

\section{Introduction}
Nonlocal minimal surfaces are the fractional counterpart of the classical minimizers of the perimeter functional. A special subclass is given by nonlocal minimal graphs, namely nonlocal minimal surfaces that possess a graphical structure.

Because of the presence of long-range interactions, nonlocal minimal graphs behave quite differently from their classical analogues, most notably through the typical appearance of boundary discontinuities. Nevertheless, nonlocal minimal graphs remain continuous from the interior up to the boundary of the reference domain, and therefore admit a {\em boundary trace}. The fundamental properties of this trace are, at present, largely unknown,
also due to the complexities created by the coexistence of horizontal and vertical
normal directions.

This paper is concerned with the regularity of this boundary trace. Our main result establishes~$C^{1,\gamma}$-regularity of the trace at points of boundary discontinuity. As a byproduct, we also obtain that nonlocal minimal graphs are differentiable
(of an order higher than that dictated by the underlying geometric operator)
at every point of boundary continuity.

\begin{figure}[htbp]
        \centering
         \includegraphics[width=0.55\linewidth]{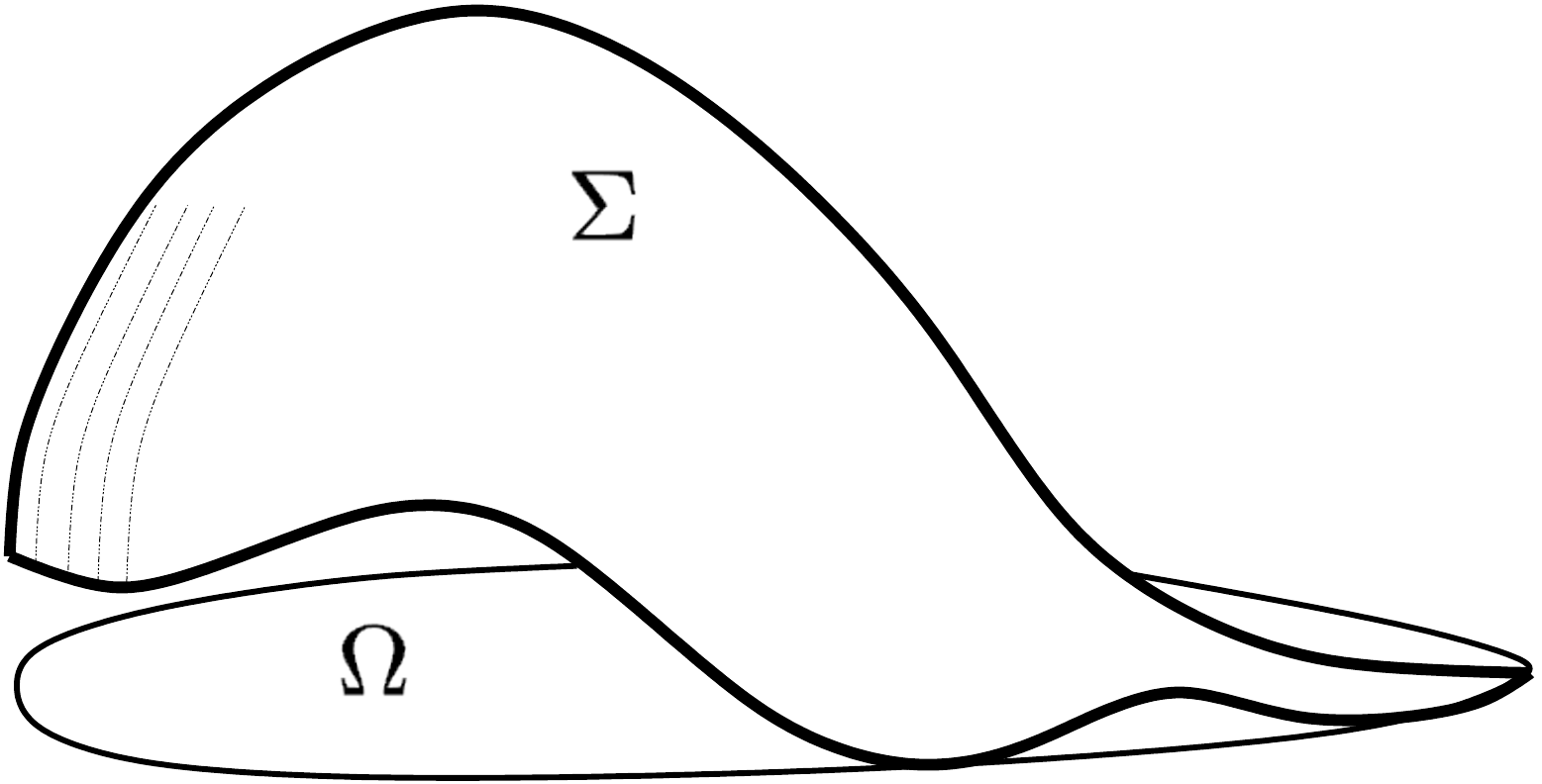}
        \caption{\footnotesize\sl Sketch illustrating a nonlocal minimal graph ``rising'' due to far-away data.
The figures in this paper serve a purely expository purpose and do not aim to capture the full complexity of nonlocal minimal graphs.}
        \label{f8103pra3g5e7.1ascvfg}
    \end{figure}
    
The precise mathematical framework that we consider is as follows (see Figure~\ref{f8103pra3g5e7.1ascvfg} for a simple diagram).
Throughout this paper, $\Omega$ will denote an open, bounded set of~$\R^n$ of
Lipschitz class
and our objective is to consider minimizers of the $s$-perimeter functional in the cylinder~$
{\mathcal{C}}:=\Omega\times\R$. For this minimization, we will consider an external datum~$E_0$ with a graphical structure. More specifically, we consider a continuous
function~$u_0:\R^n\to\R$ and define~$E_0:=\{x_{n+1}<u_0(x')\}$, where points in~$\R^{n+1}$
are denoted by~$x=(x',x_{n+1})\in\R^n\times\R$.

Given~$s\in(0,1)$ and two disjoint sets~$A$ and~$B$ in~$\R^{n+1}$, we consider the integral interaction $$I_s(A,B):=\iint_{A\times B}\frac{dx\,dy}{|x-y|^{n+1+s}}$$
and, in the wake of~\cite{MR2675483}, we define the $s$-perimeter of a (measurable)
set~$F\subseteq\R^{n+1}$ in~${\mathcal{C}}$ as
$$ {\operatorname{Per}}_s(F,{\mathcal{C}}):=I_s(F\cap{\mathcal{C}},F^c\cap {\mathcal{C}})+I_s(F\cap {\mathcal{C}},F^c\cap {\mathcal{C}}^c)+I_s(F\cap{\mathcal{C}}^c,F^c\cap {\mathcal{C}}),
$$
where the superscript~$c$ denotes the complementary set in~$\R^{n+1}$.

We let~$E$ be a minimizer for the $s$-perimeter in~${\mathcal{C}}$ with respect to
the external datum~$E_0$: more precisely, we suppose that~$E\cap{\mathcal{C}}^c=E_0\cap{\mathcal{C}}^c$ and, for every bounded Lipschitz
set~$U\subseteq{\mathcal{C}}$ and every~$F\subseteq\R^{n+1}$ with~$F\cap U^c=
E\cap U^c$, it holds that~${\operatorname{Per}}_s(E,U)\le {\operatorname{Per}}_s(F,U)$.
See~\cite{MR3827804} for more details about this notion of minimization.

It is known that~$E$ has also a graphical structure
(see~\cite{MR3516886}), namely~$E=\{x_{n+1}<u_E(x')\}$, for some~$u_E:\R^n\to\R$
with~$u_E=u_0$ in~$\R^n\setminus\Omega$. The function~$u_E$ is also known to be
uniformly continuous in~$\Omega$ (see~\cite{MR3516886})
and~$C^\infty(\Omega)$ (see~\cite[Theorem~1.1]{MR3934589}). 
In particular, the uniform continuity of~$u_E$ allows us
to look at~$u_E(\partial\Omega)$, which can be considered as the ``trace''
of the nonlocal minimal surface~$\Sigma:=\partial E$ along~$\partial{\mathcal{C}}$
(coming from the interior of~${\mathcal{C}}$).

However,
boundary jumps for~$u_E$ can occur along~$\partial\Omega$, where the derivatives
of~$u_E$ may also blow up: this phenomenon is called ``stickiness''
and it is peculiar of nonlocal minimal surfaces, see~\cite{MR3596708}
(in fact, the stickiness phenomenon is somewhat ``generic'' for 
nonlocal minimal graphs, see~\cite{MR4104542}).

Though~$u_E$ experiences jumps and derivative blow up patterns
around the boundary of~${\mathcal{C}}$, it is known that
\begin{equation}\label{REGSI}
{\mbox{the hypersurface~$\Sigma$ is of class~$C^{1,\frac{1+s}2}$ in~$\overline{\mathcal{C}}$,}}\end{equation} see~\cite[Theorem~1.1]{MR3532394}.
In this spirit, the external unit normal~$\nu$ to~$\Sigma$
is well-defined and continuous not only on~$\Sigma\cap{\mathcal{C}}$
(where it can be written as~$\frac{(-\nabla u_E,1)}{\sqrt{|\nabla u_E|^2+1}}$)
but also on~$\Sigma\cap\partial{\mathcal{C}}$
(where vertical tangents and stickiness phenomena
may produce cases in which~$\nu_{n+1}=0$).

We point out that the external normal is actually well-defined
on the whole of~$\partial^*E$ (the reduced boundary of
the set~$E$, which, by construction, is of locally finite perimeter).
This is useful, since it allows one to write a suitable nonlocal equation of geometric flavor
for~$\nu$ (as we will discuss in Section~\ref{qsojdwco34oiy0h6pym2r3p4gtgS}
below). For this, we 
denote by~${\mathcal{H}}^n$ the $n$-dimensional Hausdorff measure and
suppose that
\begin{equation}\label{ALINF} \int_{\partial^*E_0\cap{\mathcal{C}}^c}\frac{d{\mathcal{H}}^n_y}{(1+|y|)^{n+1+s}}<+\infty.\end{equation}

In this setting, the main result of this paper goes as follows:

\begin{theorem}[$C^{1,\gamma}$-regularity of the
trace of nonlocal minimal graphs at points of stickiness]\label{THM1}
Let~$x_0\in\partial\Omega$ and~$X_0:=(x_0,u(x_0))$. Suppose that~$\Omega$
is of class~$C^{2,1}$ in a neighborhood of~$x_0$ and
$$ \lim_{\Omega\ni x'\to x_0} u_E(x')\ne\lim_{\R^n\setminus\Omega\ni x'\to x_0} u_0(x').$$

Then, there exists~$\rho>0$ such that~$u_E(\partial\Omega)\cap B_\rho(X_0)$
is an $(n-1)$-dimensional surface of class $C^{1,\gamma}$
in~$\R^{n+1}$.
\end{theorem}

The statement in Theorem~\ref{THM1} is the first regularity result for the boundary trace of 
nonlocal minimal surfaces, and even the Lipschitz regularity (which is
a subcase of Theorem~\ref{THM1}) is new. 

The methodology introduced to prove Theorem~\ref{THM1} will allow us
to obtain the differentiability of nonlocal minimal graphs at points
of boundary continuity, according to the following result:

\begin{theorem}[Differentiability of nonlocal minimal graphs at points of boundary continuity]
\label{2swa1l2e4oj2dn2fl3er4t5e4n2v2}
Let~$B'_\rho$ be the ball of radius~$\rho$ in~$\R^n$, centered at the origin.

Assume that~$\Omega\cap B_1'=\{x_n>0\}\cap B_1'$, that~$u_0=0$ in~$\{x_n<0\}\cap B_1'$,
and that
\begin{equation}\label{2swa1l2e4oj2dn2fl3er4t5e4n2v}
\lim_{\Omega\ni x'\to 0} u_E(x')=0.
\end{equation}

Then, there exist~$C>0$ and~$r\in\left(0,\frac12\right)$ such that, for all~$x'\in B_r'$,
\begin{equation}\label{2swa1l2e4oj2dn2fl3er4t5e4n2v1}
|u_E(x')|\le C|x'|^{\frac{3+s}2}.
\end{equation}
\end{theorem}

We remark that assumption~\eqref{2swa1l2e4oj2dn2fl3er4t5e4n2v} says that the origin
is a boundary continuity point for the nonlocal minimal set~$E$. The thesis in~\eqref{2swa1l2e4oj2dn2fl3er4t5e4n2v1} gives that the nonlocal minimal graph~$u_E$
is differentiable, and actually of class~$C^{1,\frac{1+s}2}$, at this continuity point.

This kind of results is intriguing, since it gives that ``boundary continuity implies
boundary differentiablity''. Theorem~\ref{2swa1l2e4oj2dn2fl3er4t5e4n2v2}
was proved in~\cite[Theorem~1.4]{MR4178752} when~$n=2$. As such,
Theorem~\ref{2swa1l2e4oj2dn2fl3er4t5e4n2v2} completes this investigation
in all dimensions, as requested in~\cite[Open Problem~1.7]{MR4178752}.

Also, Theorems~\ref{THM1} and~\ref{2swa1l2e4oj2dn2fl3er4t5e4n2v2} imply that the boundary trace of nonlocal
minimal graphs do not exhibit vertical tangents
along~$\partial{\mathcal{C}}$. This solves a question posed in~\cite[Open Problem~1.6]{MR4178752}.

The proof of Theorem~\ref{THM1} relies on suitable estimates
on the ratio between tangential components of the normal
and the vertical component (for this, the graph property of the minimal set
is helpful, guaranteeing that the vertical component of the normal is positive,
avoiding zero divisors in this ratio).

The argument to prove Theorem~\ref{THM1} can actually be conveniently
split into two separate steps: first (see Section~\ref{SECLP}),
one obtains a uniform bound on the above
ratio (corresponding to a Lipschitz regularity in Theorem~\ref{THM1});
then (see Section~\ref{LAKPMOD-iq70jwohg0495ti}),
one uses this result to obtain an enhanced
estimate on the oscillation of the above ratio
(this entails differentiable regularity
in Theorem~\ref{THM1}, with a H\"older control of the oscillation
leading to the desired~$C^{1,\gamma}$-regularity).

Once Theorem~\ref{THM1} is established, Theorem~\ref{2swa1l2e4oj2dn2fl3er4t5e4n2v2}
follows from the methods introduced in~\cite{MR4178752}
and the boundary Harnack inequalities developed in this paper
(see Section~\ref{M324r3tgRguGJKHDOJCDL417872152}).

Indeed, the crucial ingredient in the bounds on the above normal ratio
consists in two new
geometric boundary Harnack inequalities.
The reason for which boundary Harnack inequalities
play a role in this setting is that,
in the vicinity of a stickiness point, the tangential components
of the normal vanish when the normal becomes horizontal (see Figure~\ref{f8103pra3g5e7})
and the significance of the boundary Harnack inequalities 
is to guarantee that the denominator of this ratio detaches
from zero sufficiently fast, due to the presence of regions in which the
hypersurface is a ``nice graph''.

\begin{figure}[htbp]
        \centering
        \includegraphics[width=0.5\linewidth]{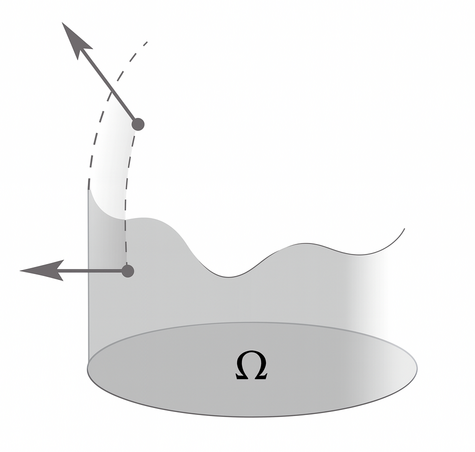}
        \caption{\footnotesize\sl Sketch of the exterior normal of a nonlocal minimal graph in the vicinity
        of a point of stickiness.}
        \label{f8103pra3g5e7}
    \end{figure}

The first geometric boundary Harnack inequality will be presented in Lemma~\ref{LE:31} and
its key feature is to obtain a uniform bound on the ratio
without any assumption on the set where the denominator vanishes.
This corresponds to a Lipschitz regularity for the trace of nonlocal
minimal graphs at points of stickiness, without any initial assumption.

The second geometric boundary Harnack inequality will be presented in Lemma~\ref{LEM5.1}.
This one leverages the result of the first, since it assumes that
the set where the denominator vanishes is already of Lipschitz class
(or, more generally, that one can find positivity sets for the denominator
at any scale with a linear structure). With this additional hypothesis,
the second geometric boundary Harnack inequality provides a H\"older control
of the oscillation of the normal ratio. This corresponds to the $C^{1,\gamma}$-regularity
for the trace of nonlocal
minimal graphs at points of stickiness, relying on the Lipschitz regularity that
was established in the previous step.

In fact, the terminology ``boundary
inequality'' is perhaps too belittling, since we will not need to assume that the solution vanishes at
the point where these inequalities are centered. In this spirit, what we actually establish
are Harnack inequalities in a geometric framework that are uniform up to the boundary
(and thus oscillation estimates for the ratio of the normal that are uniform up to the boundary).

Concerning the existing literature on boundary Harnack inequalities in the nonlocal setting, we recall that the first result of this type was established in~\cite{MR1438304} for the fractional Laplacian in Lipschitz domains. This was later extended to arbitrary open sets in~\cite{MR1719233, MR2365478}, and to symmetric Markov jump processes in~\cite{MR3271268}. The case of nonlocal operators in non-divergence form was first addressed in~\cite{MR3694738}, under suitable homogeneity assumptions on the interaction kernel and assuming that the domain is of class $C^1$. Subsequently, in~\cite{MR4023466}, boundary Harnack inequalities were proved for non-divergence form operators in general domains. The results of~\cite{MR4023466} have played a pivotal role in the study of the regularity of free boundary obstacle problems, see~\cite{MR3648978}.

Among the various proof techniques developed in the literature, the approach in~\cite{MR4023466} is arguably the closest to the one adopted here. Indeed, it is based on a general strategy consisting of two main steps: a pointwise bound in general domains, combined with a H\"older estimate in Lipschitz domains. Nevertheless, several important differences arise between
the framework of~\cite{MR4023466} and that of the present paper. First, the analysis in~\cite{MR4023466} is carried out in the Euclidean setting, whereas here we work on a hypersurface. As a consequence of this geometric context, the interaction kernel considered in~\cite{MR4023466} is symmetric, while the kernel studied here is not necessarily symmetric. Moreover, the framework in~\cite{MR4023466} is tailored to pairs of solutions satisfying the same homogeneous equation, or equations whose right-hand side is modeled on the coefficients of suitable linear combinations of the solutions. In contrast, in the present setting the right-hand side is typically non-homogeneous and must incorporate the specific terms arising from the geometric equations under consideration.

We now dive into the technical details needed to prove the main results of this paper.

\section{Lipschitz regularity of the
trace of nonlocal minimal graphs at points of stickiness}\label{SECLP}

This section contains the first step towards the regularity
theory in Theorem~\ref{THM1}, dealing with regularity in the Lipschitz class.
Though we will later improve the results to obtain $C^{1,\gamma}$-regularity,
this intermediate step is crucial, since it will allow us to deal (later on
in Section~\ref{LAKPMOD-iq70jwohg0495ti}) with equations
that hold in ``sufficiently regular'' domains.

\subsection{A cutoff version of the equation for the normal}\label{qsojdwco34oiy0h6pym2r3p4gtgS}

Hypothesis~\eqref{ALINF} allows one to use~\cite[equation~(4.3) in Proposition~4.1]{MR3934589}
in the whole of~$\R^{n+1}$, finding that, for every~$x\in\Sigma\cap{\mathcal{C}}$,
\begin{equation*}
\int_{\partial^*E}\frac{(\nu(x)-\nu(y))\cdot e}{|x-y|^{n+1+s}}\,d{\mathcal{H}}^n_y=0,
\end{equation*}for all~$e\in\R^{n+1}$
such that~$e\cdot\nu(x)=0$ (as usual, the singular integral above
needs to be interpreted in the
Cauchy principal value sense).

As a result (see~\cite[page~803]{MR3934589}), one obtains that, for every~$x\in\Sigma\cap{\mathcal{C}}$,
\begin{equation}\label{EQ:NORM}
\int_{\partial^*E}\frac{\nu(x)-\nu(y)}{|x-y|^{n+1+s}}\,d{\mathcal{H}}^n_y
=\nu(x)\int_{\partial^*E}\frac{1-\nu(x)\cdot\nu(y)}{|x-y|^{n+1+s}}\,d{\mathcal{H}}^n_y.
\end{equation}

The goal of this section is to provide 
a ``localized'' version of~\eqref{EQ:NORM},
so to separate the contribution coming from the smooth component
of~$\nu$ from the one coming from the geometric measure theory
notion of the normal.

\begin{lemma}\label{CUTLE}
Let~$x_0\in \Sigma$ and~$R>0$. Assume that
\begin{equation}\label{DOKFM}\Sigma\cap B_R(x_0)\subseteq\overline{\mathcal{C}}.\end{equation}

Then, for all~$x\in\Sigma\cap B_{R/2}(x_0)\cap{\mathcal{C}}$,
\begin{equation}\label{poqljf34o53i6ujp6ui-2pro345lty}
\int_{\Sigma\cap B_R(x_0)}\frac{\nu(y)-\nu(x)}{|x-y|^{n+1+s}}\,d{\mathcal{H}}^n_y
+\nu(x)\int_{\Sigma\cap B_R(x_0)}\frac{1-\nu(x)\cdot\nu(y)}{|x-y|^{n+1+s}}\,d{\mathcal{H}}^n_y=a(x)\nu(x)+f(x),
\end{equation}
for suitable~$a:\Sigma\cap B_{R/2}(x_0)\cap{\mathcal{C}}\to\R$ and~$f=(f_1,\dots,f_{n+1}):\Sigma\cap B_{R/2}(x_0)\cap{\mathcal{C}}\to\R^{n+1}$.

In addition, $f_{n+1}\le0$ and
\begin{equation}\label{ALINF-2} |a|+|f|\le C,\end{equation}
for a suitable~$C>0$ depending only on~$n$, $s$, $R$, and~$E_0$.
\end{lemma}

\begin{proof} We observe that~$\partial^* E\cap B_R(x_0)
=\Sigma\cap B_R(x_0)$, thanks to~\eqref{REGSI} and~\eqref{DOKFM}.
Hence,
we employ~\eqref{EQ:NORM} and, after a cancellation, find that
\begin{eqnarray*}&&
\int_{\Sigma\cap B_R(x_0)}\frac{\nu(y)-\nu(x)}{|x-y|^{n+1+s}}\,d{\mathcal{H}}^n_y
+\nu(x)\int_{\Sigma\cap B_R(x_0)}\frac{1-\nu(x)\cdot\nu(y)}{|x-y|^{n+1+s}}\,d{\mathcal{H}}^n_y\\&&\qquad=\int_{(\partial^*E)\setminus B_R(x_0)}\frac{\nu(x)-\nu(y)}{|x-y|^{n+1+s}}\,d{\mathcal{H}}^n_y
-\nu(x)\int_{(\partial^*E)\setminus B_R(x_0)}\frac{1-\nu(x)\cdot\nu(y)}{|x-y|^{n+1+s}}\,d{\mathcal{H}}^n_y\\
&&\qquad=-\int_{(\partial^*E)\setminus B_R(x_0)}\frac{\nu(y)}{|x-y|^{n+1+s}}\,d{\mathcal{H}}^n_y
+\nu(x)\int_{(\partial^*E)\setminus B_R(x_0)}\frac{\nu(x)\cdot\nu(y)}{|x-y|^{n+1+s}}\,d{\mathcal{H}}^n_y.
\end{eqnarray*}
This establishes~\eqref{poqljf34o53i6ujp6ui-2pro345lty} with
\begin{equation}\label{oqjdcmwerlbX2} f(x):=-\int_{(\partial^*E)\setminus B_R(x_0)}\frac{\nu(y)}{|x-y|^{n+1+s}}\,d{\mathcal{H}}^n_y
\end{equation}
and$$
a(x):=\int_{(\partial^*E)\setminus B_R(x_0)}\frac{\nu(x)\cdot\nu(y)}{|x-y|^{n+1+s}}\,d{\mathcal{H}}^n_y.
$$

Notice that~$f_{n+1}\le0$ since~$\nu_{n+1}\ge0$. Furthermore,
for all~$x\in \Sigma\cap B_{R/2}(x_0)$ and~$y\in (\partial^*E)\setminus B_R(x_0)$,
we have that~$|x-y|\ge |y-x_0|-|x-x_0|\ge\frac{|y-x_0|}2$
and so~\eqref{ALINF-2} follows from~\eqref{ALINF}
and the interior regularity in~\cite{MR3934589}.
\end{proof}

It will also be convenient to have a ``commutator estimate'' on the fractional operator
appearing in~\eqref{poqljf34o53i6ujp6ui-2pro345lty}. This is provided by the following result:

\begin{lemma}\label{COMMUTATOR}
Let~$\alpha\in(s,1]$.
Let~${\mathcal{S}}$ be a hypersurface of class~$C^{1,\alpha}$ and the linear operator (in the Cauchy principal value sense)
$$ {\mathcal{L}}f(x):=\int_{{\mathcal{S}}\cap B_1}\frac{f(y)-f(x)}{|x-y|^{n+1+s}}\,d{\mathcal{H}}^n_y.$$

Then,
for all~$\rho>0$, all~$f$, $g\in C^{1,\alpha}( {\mathcal{S}}\cap B_1)$ and all~$x\in{\mathcal{S}}$ such that the geodesic ball of radius~$\rho$ centered at~$x$
is contained in~${\mathcal{S}}\cap B_1$, we have that
$$ |g(x){\mathcal{L}}f(x)-{\mathcal{L}}(fg)(x)|\le
C\,
\|f\|_{C^{0,\alpha}( {\mathcal{S}}\cap B_1)}
\|g\|_{C^{1,\alpha}( {\mathcal{S}}\cap B_1)},$$
for some~$C>0$ depending only on~$n$, $s$, $\alpha$, $\rho$, and~${{\mathcal{S}}}$.
\end{lemma}

\begin{proof} We point out that
\begin{equation}\label{oqjdwlertg3prif9203ijksws0}\begin{split}&
\int_{{\mathcal{S}}\cap B_1}\frac{|f(x)-f(y)|\,|g(x)-g(y)|}{|x-y|^{n+1+s}}\,d{\mathcal{H}}^n_y\\&\qquad\le
\|f\|_{C^{0,\alpha}( {\mathcal{S}}\cap B_1)}\|g\|_{C^{0,1}( {\mathcal{S}}\cap B_1)}
\int_{{\mathcal{S}}\cap B_1}\frac{d{\mathcal{H}}^n_y}{|x-y|^{n+s-\alpha}}\\&\qquad\le C_1\,\|f\|_{C^{0,\alpha}( {\mathcal{S}}\cap B_1)}\|g\|_{C^{0,1}( {\mathcal{S}}\cap B_1)},\end{split}
\end{equation}for some~$C_1>0$ depending only on~$n$, $s$, and~${{\mathcal{S}}}$.

Moreover, if~${{\mathcal{S}}}$ is parameterized by some map~$\Phi:\R^n\to\R^{n+1}$,
using the notation~$\underline g=g\circ\Phi$, $
x=\Phi(\underline x)$, and~$y=\Phi(\underline y)$ we have that
\begin{equation}\label{opqlsjdwfworltgmb2134t5y67u23ed.2}
g(y)-g(x)=\underline g(\underline y)-\underline g(\underline x)=
\nabla\underline g(\underline x)\cdot(\underline y-\underline x)+\xi(\underline x,\underline y),
\end{equation}
with
\begin{equation}\label{opqlsjdwfworltgmb2134t5y67u23ed.3}|\xi(\underline x,\underline y)|\le C_2\,[g]_{C^{1,\alpha}( {\mathcal{S}}\cap B_1)}
|\underline x-\underline y|^{\alpha+1},\end{equation}
for some~$C_2>0$ depending only on~$n$ and~${{\mathcal{S}}}$.

Consequently, if~${\mathcal{S}}\cap B_1=\Phi(U)$ for some~$U\subseteq\R^n$,
denoting by~$J$ the surface density of~${\mathcal{S}}$
(which can be obtained via the
sum of squares of the subdeterminants of the Jacobian of~$\Phi$, see e.g.~\cite[pages~101-102]{MR1158660}), we see that
\begin{equation}\label{OSJ324T0xY56P7OU-PKRFOaosd}\begin{split}&\left|
\int_{{\mathcal{S}}\cap B_1}\frac{g(x)-g(y)}{|x-y|^{n+1+s}}\,d{\mathcal{H}}^n_y\right|=
\left|
\int_{U}\frac{\underline g(\underline y)-\underline g(\underline x)}{|\Phi(\underline x)-\Phi(\underline y)|^{n+1+s}}\,J(\underline y)\,d\underline y\right|\\&\qquad\le
\left|
\int_{U}\frac{\nabla\underline g(\underline x)\cdot(\underline y-\underline x)}{|\Phi(\underline x)-\Phi(\underline y)|^{n+1+s}}\,J(\underline y)\,d\underline y\right|+C_3\,[g]_{C^{1,\alpha}( {\mathcal{S}}\cap B_1)}\\&\qquad\le
\left|
\int_{U}\frac{\nabla\underline g(\underline x)\cdot(\underline y-\underline x)}{|\Phi(\underline x)-\Phi(\underline y)|^{n+1+s}}\,J(\underline x)\,d\underline y\right|+C_4\,[g]_{C^{1,\alpha}( {\mathcal{S}}\cap B_1)},
\end{split}
\end{equation}
for some~$C_3$, $C_4>0$ depending\footnote{For further reference,
we also note that if~$g\in C^{1,1}(B_1)$, one can also
replace here~$[g]_{C^{1,\alpha}( {\mathcal{S}}\cap B_1)}$
with~$[g]_{C^{1,1}( B_1)}$.
This is due to the fact that, in this situation, one can replace~\eqref{opqlsjdwfworltgmb2134t5y67u23ed.2}
and~\eqref{opqlsjdwfworltgmb2134t5y67u23ed.3}
with\label{opqlsjdwfworltgmb2134t5y67u23ed.32ojwfvl}
$$ g(y)-g(x)=
\nabla g(x)\cdot( y- x)+\xi(x, y)=
\nabla g(\Phi(\underline x))\cdot( \Phi(\underline y)- \Phi(\underline x)\big)+\xi( x,y)
,$$with $$|\xi(x, y)|\le C_2\,[g]_{C^{1,1}(  B_1)}
|x-y|^2.$$} only on~$n$, $s$, $\alpha$,
and the~$C^{1,\alpha}$-regularity
of~${{\mathcal{S}}}$.

To ease the notation, we now suppose, up to a rigid motion, that~$\underline x=0$. In this way, $\Phi(\underline y)=x+D\Phi(0)\underline y+O(|\underline y|^{\alpha+1})$ and therefore
$$\frac1{|x-\Phi(-\underline y)|^{n+1+s}}=\frac{1+O(
|\underline y|^{\alpha})
}{|x-\Phi(\underline y)|^{n+1+s}}.$$

We also recall that,
by construction, $\{|\underline y|\le \rho/C_5\}\subseteq U$ for some~$C_5>0$ depending only on~$n$ and
the~$C^{1,\alpha}$-regularity
of~${{\mathcal{S}}}$.
As a result,
\begin{eqnarray*}&& 2\int_{\{|\underline y|\le \rho/C_5\}}\frac{ \underline y}{|x-\Phi(\underline y)|^{n+1+s}}\,d\underline y\\&&\qquad
=\int_{\{|\underline y|\le \rho/C_5\}}\frac{ \underline y}{|x-\Phi(\underline y)|^{n+1+s}}\,d\underline y
-\int_{\{|\underline y|\le \rho/C_5\}}\frac{ \underline y}{|x-\Phi(-\underline y)|^{n+1+s}}\,d\underline y\\
&&\qquad=\int_{\{|\underline y|\le \rho/C_5\}}\frac{ O(|\underline y|^{\alpha+1})
}{|x-\Phi(\underline y)|^{n+1+s}}\,d\underline y\\&&\qquad=O(\rho^{\alpha-s}),
\end{eqnarray*}
yielding that
$$ \left|
\int_{U}\frac{\underline y}{|x-\Phi(\underline y)|^{n+1+s}}\,d\underline y\right|\le
\left|
\int_{U\cap\{|\underline y|> \rho/C_5\}}\frac{\underline y}{|x-\Phi(\underline y)|^{n+1+s}}\,d\underline y\right|+C_6\le C_7,$$
for some~$C_6$, $C_7>0$ depending only on~$n$, $s$, $\alpha$, $\rho$, and
the~$C^{1,\alpha}$-regularity
of~${{\mathcal{S}}}$.

This and~\eqref{OSJ324T0xY56P7OU-PKRFOaosd} give that
\begin{equation}\label{OSJ324T0xY56P7OU-PKRFOaosd:aaihsdjdvl}\begin{split}&\left|
\int_{{\mathcal{S}}\cap B_1}\frac{g(x)-g(y)}{|x-y|^{n+1+s}}\,d{\mathcal{H}}^n_y\right|\le
C_8\,[g]_{C^{1,\alpha}( {\mathcal{S}}\cap B_1)},
\end{split}
\end{equation}
for some~$C_8>0$ depending only on~$n$, $s$, $\alpha$,
$\rho$, and
the~$C^{1,\alpha}$-regularity
of~${{\mathcal{S}}}$.

Now we remark that
\begin{eqnarray*}&&
g(x){\mathcal{L}}f(x)-{\mathcal{L}}(fg)(x)\\&&\qquad
=g(x)\int_{{\mathcal{S}}\cap B_1}\frac{f(y)-f(x)}{|x-y|^{n+1+s}}\,d{\mathcal{H}}^n_y
-\int_{{\mathcal{S}}\cap B_1}\frac{(fg)(y)-(fg)(x)}{|x-y|^{n+1+s}}\,d{\mathcal{H}}^n_y\\
&&\qquad=
\int_{{\mathcal{S}}\cap B_1}\frac{f(y)(g(x)-g(y))}{|x-y|^{n+1+s}}\,d{\mathcal{H}}^n_y\\&&\qquad=
\int_{{\mathcal{S}}\cap B_1}\frac{(f(y)-f(x))(g(x)-g(y))}{|x-y|^{n+1+s}}\,d{\mathcal{H}}^n_y
+f(x)
\int_{{\mathcal{S}}\cap B_1}\frac{g(x)-g(y)}{|x-y|^{n+1+s}}\,d{\mathcal{H}}^n_y.
\end{eqnarray*}
{F}rom this, \eqref{oqjdwlertg3prif9203ijksws0} and~\eqref{OSJ324T0xY56P7OU-PKRFOaosd:aaihsdjdvl},
the desired result follows.
\end{proof}

\subsection{A geometric boundary Harnack inequality}

We provide here a boundary Harnack inequality
of nonlocal type in a geometric setting, that will be used to deduce a Lipschitz control on~$\Sigma$ in the vicinity
of boundary discontinuity points.

\begin{lemma}\label{LE:31}
Let~$\alpha\in(0,1)$, $c_0>0$, and~$C_0$, $C_1\ge0$.

Let~${\mathcal{S}}\subseteq\R^{n+1}$ be a bounded, connected hypersurface of class~$C^{1,\alpha}$
such that~$0\in{\mathcal{S}}$. Let~$
{\mathcal{D}}$ be an open subset of~${\mathcal{S}}$ (in its relative topology).

Consider a kernel~${\mathcal{K}}:{\mathcal{S}}\times{\mathcal{S}}\to[0,+\infty]$
such that
\begin{equation}\label{oqwekfjug54toXypuk} \inf_{|x-y|\le2}{\mathcal{K}}(x,y)\ge c_0\end{equation}
and the linear operator (in the Cauchy principal value sense)
\begin{equation}\label{osjkdwfrjptykkhjxi19iu24ojt45l}
{\mathcal{L}} f(x):=\int_{{\mathcal{S}}\cap B_1}\big( f(y)-f(x)\big)\,{\mathcal{K}}(x,y)\,d{\mathcal{H}}^n_y,\end{equation} 
satisfying, for all~$f\in C^{1,\alpha}({\mathcal{S}}\cap B_1)$,
\begin{equation}\label{oqwekfjug54toXypuk:1a}
|{\mathcal{L}} f|\le
C_1\,\|f\|_{C^{1,\alpha}({\mathcal{S}}\cap B_1)}.
\end{equation}

Let~$u$, $v:{\mathcal{S}}\to\R$, with~$u\ge0$ and~$v\le1$ in~${\mathcal{S}}\cap B_1$.

Assume that\begin{equation}\label{1-SIPDOQ3RL24RTY0XUP.99}
({\mathcal{S}}\setminus
{\mathcal{D}})\cap B_1\subseteq
\{u=0\}\cap\{v=0\}.\end{equation}
Assume also that
\begin{equation}\label{oqw5re32s1ekfju5rfc9iujhb5tyhn90j3g54toXypuk:1a}{\mbox{${\mathcal{L}}u\le \ell u$ and~${\mathcal{L}}v\ge\ell v- C_0$ in~${\mathcal{D}}\cap B_{1/2}$,}}\end{equation} for some~$\ell\in L^\infty({\mathcal{S}}\cap B_1)$.

Suppose in addition that there exist~$\rho_0\in(0,1)$,
$p_0\in {\mathcal{S}}\cap (B_1\setminus B_{\rho_0})$,
$r_0\in\left(0,\frac{1-|p_0|-\rho_0}8\right)$, and~$\delta_0>0$ such that
\begin{equation}\label{1-SIPDOQ3RL24RTY0XUP} \inf_{B_{r_0}(p_0)}u\ge\delta_0.\end{equation}

Then,
there exists~$C>0$, depending only on~$n$, $\alpha$, ${\mathcal{L}}$, the $C^{1,\alpha}$-regularity of~${\mathcal{S}}$, $c_0$, $C_0$, $C_1$, $\|\ell\|_{ L^\infty({\mathcal{S}}\cap B_1)}$,
$\rho_0$, $r_0$,
and~$\delta_0$ (and independend of~$
{\mathcal{D}}$), such that
$$ v\le C u {\mbox{ in }}{\mathcal{S}}\cap B_{\rho_0/2}.$$
\end{lemma}

\begin{proof} We let~$\delta\in\left(0,\frac{\delta_0}2\right)$, to be chosen conveniently small in what follows, and~$\eta:=u-\delta v$. 
We stress that~$\eta\ge-\delta$ in~${\mathcal{S}}\cap B_1$.
Let also~$\rho\in\left(0,\frac{\rho_0}2\right]$.
Suppose that there exists~$p\in{\mathcal{S}}\cap B_{\rho}$ such that~$\eta(p)<0$
and, given~$\lambda\in\R$, consider the family of functions
$${\mathcal{S}}\ni x\longmapsto\phi_\lambda(x):=-\frac{64\delta}{(\rho-\rho_0)^2} |x-p|^2-\lambda.
$$
Notice that~$\phi_\lambda\le-\lambda<-\delta\le\eta$ in~${\mathcal{S}}\cap B_1$ as long as~$\lambda>\delta$, but~$\phi_\lambda(p)=-\lambda>0>\eta(p)$ as long as~$\lambda<0$.

Hence, we can choose~$\lambda\in[0,\delta]$ such that~$\phi_\lambda\le\eta$ in~${\mathcal{S}}\cap B_1$
and there exists~$q\in{\mathcal{S}}\cap B_1$ such that~$\phi_\lambda(q)=\eta(q)$.

We have that~$q\in{\mathcal{D}}$, otherwise, by~\eqref{1-SIPDOQ3RL24RTY0XUP.99},
one would have that~$u(q)=v(q)=0$, and so~$\eta(q)=0$, yielding that~$\phi_\lambda(q)=0$; but this would give that~$\lambda=0$ and~$q=p$,
against the condition that~$\eta(p)<0$.

We also stress that~$q\in\overline{ B_{(\rho-\rho_0)/8}(p)}\subseteq B_{\rho}$, because
$$ \frac{64\delta}{(\rho-\rho_0)^2} |q-p|^2\le\frac{64\delta}{(\rho-\rho_0)^2} |q-p|^2+\lambda=-\phi_\lambda(q)=-\eta(q)\le\delta.$$
Thus, setting~$\psi_\lambda:=\eta-\phi_\lambda$, we have that~$\psi_\lambda\ge0=\psi_\lambda(q)$
in~${\mathcal{S}}\cap B_1$, whence, recalling~\eqref{oqwekfjug54toXypuk},
\begin{equation}\label{oqwekfjug54toXypuk:1ab}
\begin{split}&
c_0\int_{\mathcal{S}\cap B_{r_0}(p_0)}\psi_\lambda(y)\,d{\mathcal{H}}^n_y\le
\int_{\mathcal{S}\cap B_{r_0}(p_0)}\psi_\lambda(y)\,{\mathcal{K}}(q,y)\,d{\mathcal{H}}^n_y\\&\qquad\le
\int_{{\mathcal{S}}\cap B_1}\psi_\lambda(y)\,{\mathcal{K}}(q,y)\,d{\mathcal{H}}^n_y=
{\mathcal{L}}\psi_\lambda(q)\le \ell(q)\eta(q)+C_0\delta+|{\mathcal{L}}\phi_\lambda(q)|.\end{split}\end{equation}

Now we point out that
\begin{equation}\label{osjqwdf93m7i0un6pu8Li-porkgb}
|\eta(q)|=|\phi_\lambda(q)|\le
\frac{64\delta}{(\rho-\rho_0)^2} |q-p|^2+\lambda\le
\delta+\delta=2\delta.
\end{equation}

Notice also that, in view of~\eqref{oqwekfjug54toXypuk:1a}, for all~$x\in{\mathcal{S}}\cap B_1$,
\begin{equation*}\frac{(\rho-\rho_0)^2\,{|\mathcal{L}}\phi_\lambda(x)|}{64\delta}\le
\int_{{\mathcal{S}}\cap B_1}\big| |x-p|^2-|y-p|^2\big|\,{\mathcal{K}}(x,y)\,d{\mathcal{H}}^n_y\le C_2,
\end{equation*}for some~$C_2$ proportional to~$C_1$
(the proportionality factor depending on~$n$, $\alpha$, and the $C^{1,\alpha}$-regularity of~${\mathcal{S}}$).

This, \eqref{oqwekfjug54toXypuk:1ab}, and~\eqref{osjqwdf93m7i0un6pu8Li-porkgb} lead to
\begin{equation}\label{oqwekfjug54toXypuk:1abc}
\int_{\mathcal{S}\cap B_{r_0}(p_0)}\psi_\lambda(y)\,d{\mathcal{H}}^n_y\le C_3\delta,\end{equation}for some~$C_3>0$.

Now we observe that, for all~$x\in{\mathcal{S}}\cap B_1$,
we have that~$\phi_\lambda(x)\le 0$. Hence, we
recall~\eqref{1-SIPDOQ3RL24RTY0XUP} and we see that
\begin{eqnarray*}&&
\int_{\mathcal{S}\cap B_{r_0}(p_0)}\psi_\lambda(y)\,d{\mathcal{H}}^n_y\ge
\int_{\mathcal{S}\cap B_{r_0}(p_0)}\eta(y)\,d{\mathcal{H}}^n_y\ge
\int_{\mathcal{S}\cap B_{r_0}(p_0)}(\delta_0-\delta)\,d{\mathcal{H}}^n_y\\&&\qquad\ge
\frac{\delta_0}2{\mathcal{H}}^n\big(\mathcal{S}\cap B_{r_0}(p_0)\big)=:C_4.
\end{eqnarray*}
Combining this with~\eqref{oqwekfjug54toXypuk:1abc}, it follows that~$\delta\ge\frac{C_4}{C_3}$.
This condition is violated if we choose~$\delta$ sufficiently small, thus showing that~$u\ge\delta v$ in~${\mathcal{S}}\cap B_{\rho}$ as long as~$\delta$ is small enough. 
\end{proof}

\begin{corollary}\label{ojsndofkw34rtgoCIr7}
Let~$0\in \Sigma\cap\partial{\mathcal{C}}$. Assume that~$\Sigma\cap B_1\subseteq\overline{\mathcal{C}}$.
Let~$\tau\in C^{1,1}(\R^{n+1},\R^{n+1})$ and suppose that~$\tau$ is a unit tangent vector field
along~$\partial{\mathcal{C}}$.

Then, there exists~$C>0$, depending only on~$n$, $s$, $\Sigma$, $E_0$, and~$\|\tau\|_{C^{1,1}(\Sigma\cap B_1)}$, such that
$$ |\tau\cdot\nu|\le C \nu_{n+1} {\mbox{ in }}\Sigma\cap B_{1/2}.$$
\end{corollary}

\begin{proof} We want to apply Lemma~\ref{LE:31}
with~${\mathcal{S}}:=\Sigma$, ${\mathcal{D}}:=\Sigma\cap{\mathcal{C}}$, $u:=\nu_{n+1}$, $v:=\tau\cdot\nu$, $
{\mathcal{K}}(x,y):=\frac1{|x-y|^{n+1+s}}$, and the corresponding operator~${\mathcal{L}}$
as in~\eqref{osjkdwfrjptykkhjxi19iu24ojt45l}
(and notice that~\eqref{oqwekfjug54toXypuk} is automatically verified with this choice).

We stress that assumption~\eqref{oqwekfjug54toXypuk:1a} is satisfied, thanks to~\eqref{OSJ324T0xY56P7OU-PKRFOaosd:aaihsdjdvl}.

We also observe that whenever~$\nu_{n+1}(x)=0$, we have that~$x\in\partial{\mathcal{C}}$
and therefore~$\tau(x)\cdot\nu(x)=0$. Accordingly,
hypothesis~\eqref{1-SIPDOQ3RL24RTY0XUP.99} is fulfilled in this setting.
Additionally, hypothesis~\eqref{1-SIPDOQ3RL24RTY0XUP} is a consequence of the graphical
structure of~$\Sigma$ in~${\mathcal{C}}$.

Now, recalling
Lemma~\ref{CUTLE} (used here with~$x_0:=0$ and~$R:=1$), we write that
\begin{equation}\label{POajlsqwr9ghiypj}
{\mathcal{L}}\nu(x)
+\nu(x)\int_{\Sigma\cap B_1}\frac{1-\nu(x)\cdot\nu(y)}{|x-y|^{n+1+s}}\,d{\mathcal{H}}^n_y=a(x)\nu(x)+f(x).
\end{equation}
In particular,
\begin{equation}\label{oqljsmdotgRfTPkswperf2-o3pr4r.01}\begin{split}
{\mathcal{L}}\nu_{n+1}(x)&\le
{\mathcal{L}}\nu_{n+1}(x)
+\nu_{n+1}(x)\int_{\Sigma\cap B_1}\frac{1-\nu(x)\cdot\nu(y)}{|x-y|^{n+1+s}}\,d{\mathcal{H}}^n_y\\&=a(x)\nu_{n+1}(x)+f_{n+1}(x)\\&\le a(x)\nu_{n+1}(x).\end{split}
\end{equation}

We also deduce from Lemma~\ref{COMMUTATOR} that, for all~$x\in\Sigma\cap B_{3/4}$,
$$ |\tau(x)\cdot{\mathcal{L}}\nu(x)-{\mathcal{L}}(\tau\cdot\nu)(x)|\le
C_1,$$
for some~$C_1>0$ depending only on~$n$, $s$, $\Sigma$, and~$\|\tau\|_{C^{1,1}(\Sigma\cap B_1)}$
(we stress that the norm of~$\nu$ in~$C^{\frac{1+s}2}$ is bounded uniformly,
thanks to~\cite{MR3532394}).

This and~\eqref{POajlsqwr9ghiypj} yield that, for all~$x\in\Sigma\cap B_{3/4}$,
\begin{equation}\label{oqljsmdotgRfTPkswperf2-o3pr4r.02}\begin{split}&\!\!\!\!\!\!
{\mathcal{L}}(\tau\cdot\nu)(x)\\&\ge
\tau(x)\cdot{\mathcal{L}}\nu(x)-C_1\\&=
a(x)\tau(x)\cdot\nu(x)+
\tau(x)\cdot f(x)-\tau(x)\cdot\nu(x)\int_{\Sigma\cap B_1}\frac{1-\nu(x)\cdot\nu(y)}{|x-y|^{n+1+s}}\,d{\mathcal{H}}^n_y-C_1\\&\ge
a(x)\tau(x)\cdot\nu(x)-C_2,\end{split}
\end{equation}
for some~$C_2>0$ depending only on~$n$, $s$, $\Sigma$, $E_0$, and~$\|\tau\|_{C^{1,1}(\Sigma\cap B_1)}$.

The inequalities in~\eqref{oqljsmdotgRfTPkswperf2-o3pr4r.01} and~\eqref{oqljsmdotgRfTPkswperf2-o3pr4r.02} allow us to use
Lemma~\ref{LE:31} and obtain that~$|
\tau\cdot\nu|\le C \nu_{n+1}$ in~$\Sigma\cap B_{1/2}$.
{F}rom this, and possibly swapping~$\tau$ with~$-\tau$,
the desired result follows.
\end{proof}

We observe that Corollary~\ref{ojsndofkw34rtgoCIr7},
in tandem with Lemma~\ref{L2sw2I2C32rA},
suffices to prove that~$u_E(\partial\Omega)\cap B_\rho(X_0)$
in Theorem~\ref{THM1}
is an $(n-2)$-dimensional surface of Lipschitz class. We will however
improve this regularity in the forthcoming pages.

\section{$C^{1,\gamma}$-regularity of the
trace of nonlocal minimal graphs at points of stickiness and proof of Theorem~\ref{THM1}}\label{LAKPMOD-iq70jwohg0495ti}

In this section, we present the proof of Theorem~\ref{THM1}.

\subsection{Weighted averages on annuli}

An essential ingredient to prove the~$C^{1,\alpha}$-regularity of the boundary trace
is to control the weighted averages of the vertical component of the normal
of the surface on concentric annuli. To this end, we present the following result:

\begin{lemma}\label{PplemdfltX34tyh.2}
Let~$\alpha\in(s,1]$ and~${\mathcal{S}}\subseteq\R^{n+1}$ be a bounded, connected hypersurface of class~$C^{1,\alpha}$
such that~$0\in{\mathcal{S}}$. 

Given a set~$V\subseteq\R^{n+1}$, let
$${\mathcal{L}}_V f(x):=\int_{{\mathcal{S}}\cap V}\frac{ f(y)-f(x)}{|x-y|^{n+1+s}}\,d{\mathcal{H}}^n_y.$$

Let~$u:{\mathcal{S}}\to\R$, with~$u\ge0$ in~${\mathcal{S}}\cap B_2$.

Assume that
\begin{equation}\label{o1edRsjdepwrgk4pyuj203ikrtgpthyuj}
{\mathcal{L}}_{B_1} u\le Mu\quad{\mbox{ in }}
{\mathcal{S}}\cap B_1\cap\{u>0\},\end{equation}
for some~$M>0$.

Assume additionally that there exists~$\delta\in\left(0,\frac1{10}\right)$ such that,
for every~$r\in\left(0,\frac12\right)$,
\begin{equation*}
{\mbox{there exists~$z_r\in{\mathcal{S}}\cap (B_{9r/10}\setminus B_{7r/10})$
such that~${\mathcal{S}}\cap B_{\delta r}(z_r)\subseteq\{u>0\}$.}}\end{equation*}

Then,
there exist~$r_\star\in(0,1)$,
depending only on~$n$, 
$s$, $\alpha$,
${{\mathcal{S}}}$, $\delta$, and~$M$, and~$
C_\star>0$,
depending only on~$n$, 
$s$, $\alpha$,
${{\mathcal{S}}}$, and~$\delta$ (but independent of~$M$), such that, for all~$r\in(0,r_\star]$,
\begin{equation}\label{oq2jfgmoptyhlp32ktgkjogtn4839n19vib6iyoy} \int_{{\mathcal{S}}\cap (B_1\setminus B_r)}\frac{ u(y)}{|y|^{n+1+s}}\,d{\mathcal{H}}^n_y \le
C_\star \int_{{\mathcal{S}}\cap (B_{r}\setminus B_{r/2})} \frac{ u(y)}{|y|^{n+1+s}}\,d{\mathcal{H}}^n_y.\end{equation}
\end{lemma}

\begin{proof} Let~$\zeta\in {\mathcal{S}}\cap B_{\delta r/10}(z_r)$.
Given $\lambda\in[0,+\infty)$, we define
\begin{equation*} \begin{split}P_\lambda(x)=
\begin{dcases}
\lambda\left(\frac{\delta^2r^2}4-|x-\zeta|^2\right)
& {\mbox{ if }}x\in B_{\delta r}(\zeta),\\
-u(\zeta) & {\mbox{ if }}x\in{\mathcal{S}}\setminus B_{\delta r}(\zeta).\end{dcases}\end{split}\end{equation*}
When~$\lambda=0$, we have that~$ P_\lambda\le u$ in~${\mathcal{S}}\cap B_1$.
Hence, we can take the largest~$\lambda\ge0$
for which~$ P_\lambda\le u$ in~${\mathcal{S}}\cap B_1$.
This produces a touching point~$x_\star\in B_{\delta r/2}(\zeta)$.

As a result,
\begin{equation*}\begin{split}&
{\mathcal{L}}_{B_1}(u-P_\lambda)(x_\star)=
\int_{{\mathcal{S}}\cap B_1}\frac{ (u-P_\lambda)(y)}{|x_\star-y|^{n+1+s}}\,d{\mathcal{H}}^n_y\ge
\int_{{\mathcal{S}}\cap (B_1\setminus B_r)}\frac{ (u-P_\lambda)(y)}{|x_\star-y|^{n+1+s}}\,d{\mathcal{H}}^n_y\\&\qquad\qquad\qquad\ge
\int_{{\mathcal{S}}\cap (B_1\setminus B_r)}\frac{ u(y)}{|x_\star-y|^{n+1+s}}\,d{\mathcal{H}}^n_y.
\end{split}
\end{equation*}

Since, when~$y\in\R^{n+1}\setminus B_r$,
$$|x_\star-y|\le|x_\star-\zeta|+|\zeta|+|y|\le\frac{\delta r}2+\frac{r}2+|y|\le r+|y|\le2|y|,$$
we conclude that
\begin{equation}\label{o1edRsjdepwrgk4pyuj203ikrtgpthyuj2}
{\mathcal{L}}_{B_1}(u-P_\lambda)(x_\star)\ge\frac1{2^{n+1+s}}\int_{{\mathcal{S}}\cap (B_1\setminus B_r)}\frac{ u(y)}{|y|^{n+1+s}}\,d{\mathcal{H}}^n_y.\end{equation}

It is useful to observe that
\begin{equation}\label{oljsdmweINqwdfvLCYH:abnj8} \frac{\lambda\delta^2r^2}4=P_\lambda(\zeta)\le u(\zeta).\end{equation}

Now we define
$$ g(x):=P_\lambda\left(x_\star+\frac{\delta rx}{10}\right)$$
and we remark that~${\mathcal{S}}_r:=\frac{10({\mathcal{S}}-x_\star)}{\delta r}$
is a zoom-in of~${\mathcal{S}}$ (hence, its
local $C^{1,\alpha}$-regularity is controlled by the one of~${\mathcal{S}}$).

Thus, recalling~\eqref{OSJ324T0xY56P7OU-PKRFOaosd:aaihsdjdvl} and footnote~\ref{opqlsjdwfworltgmb2134t5y67u23ed.32ojwfvl},
\begin{equation}\label{qowjhdfnt034rgrbRF.3rtgrbfRb0pjmfewvoUhnhyshnd}
\begin{split}&
|{\mathcal{L}}_{B_{\delta r/10}(x_\star)}P_\lambda (x_\star)|
=\left|
\int_{{\mathcal{S}}\cap B_{\delta r/10}(x_\star)}\frac{ P_\lambda(y)-P_\lambda(x_\star)}{|x_\star-y|^{n+1+s}}\,d{\mathcal{H}}^n_y
\right|
\\&\quad=\frac{10^{1+s}}{\delta^{1+s}r^{1+s}}\left|
\int_{{\mathcal{S}}_r\cap B_1}\frac{g(x)-g(0)}{|x|^{n+1+s}}\,d{\mathcal{H}}^n_x
\right|\le \frac{C_0\,
[g]_{C^{1,1}( B_1)}}{\delta^{1+s}r^{1+s}}\\&\quad
\le
\frac{C_0\,\delta^2\,r^2\,
[P_\lambda]_{C^{1,1}( B_{\delta r/10}(x_\star))}}{\delta^{1+s}r^{1+s}}
\le
\frac{C_0\,\delta^2\,r^2\,
[P_\lambda]_{C^{1,1}( B_{3\delta r/5}(\zeta))}}{\delta^{1+s}r^{1+s}}\\&\quad
\le \frac{C_0\,\lambda\delta^2r^2}{\delta^{1+s}r^{1+s}},
\end{split}\end{equation}for some~$C_0>0$ depending only on~$n$, $s$, and~${{\mathcal{S}}}$,
and possibly varying at each step of the calculation.

Therefore, in light of~\eqref{oljsdmweINqwdfvLCYH:abnj8},
\begin{equation}\label{osjdweighTGUBHJSMDwfle098ytrdfnoJA23fegKS}
|{\mathcal{L}}_{B_{\delta r/10}(x_\star)}P_\lambda (x_\star)|\le \frac{C_0\,u(\zeta)}{\delta^{1+s}r^{1+s}}.
\end{equation}

Furthermore, for all~$y\in{\mathcal{S}}\cap B_1$, recalling~\eqref{oljsdmweINqwdfvLCYH:abnj8} we have that\begin{eqnarray*}
-P_\lambda(y)&\le&
\begin{dcases}
\lambda \delta^2r^2
& {\mbox{ if }}y\in B_{\delta r}(\zeta),\\
u(\zeta) & {\mbox{ if }}x\in{\mathcal{S}}\cap (B_1\setminus B_{\delta r}(\zeta))\end{dcases}\\&\le&
\begin{dcases}
4u(\zeta)
& {\mbox{ if }}y\in B_{\delta r}(\zeta),\\u(\zeta) & {\mbox{ if }}x\in{\mathcal{S}}\cap (B_1\setminus B_{\delta r}(\zeta))\end{dcases}\\&\le&4u(\zeta).
\end{eqnarray*}
As a consequence,
\begin{equation*}\begin{split}&
{\mathcal{L}}_{B_1\setminus B_{\delta r/10}(x_\star)}P_\lambda (x_\star)\\&\quad
=\int_{{\mathcal{S}}\cap (B_1\setminus B_{\delta r/10}(x_\star))}\frac{P_\lambda(y)-P_\lambda(x_\star)}{|x_\star-y|^{n+1+s}}\,d{\mathcal{H}}^n_y\\&\quad\ge
-(4u(\zeta)+P_\lambda(x_\star))\int_{{\mathcal{S}}\cap (B_1\setminus B_{\delta r/10}(x_\star))}\frac{d{\mathcal{H}}^n_y}{|x_\star-y|^{n+1+s}}\\&\quad=
-(4u(\zeta)+u(x_\star))\int_{{\mathcal{S}}\cap (B_1\setminus B_{\delta r/10}(x_\star))}\frac{d{\mathcal{H}}^n_y}{|x_\star-y|^{n+1+s}}\\&\quad\ge-\frac{C_0\,(u(\zeta)+u(x_\star))}{\delta^{1+s}r^{1+s}}.
\end{split}\end{equation*}

Combining this observation and~\eqref{osjdweighTGUBHJSMDwfle098ytrdfnoJA23fegKS}, we get that
\begin{equation*}
{\mathcal{L}}_{B_1}P_\lambda (x_\star)\ge-\frac{C_0\,(u(\zeta)+u(x_\star))}{\delta^{1+s}r^{1+s}}.\end{equation*}
{F}rom this, \eqref{o1edRsjdepwrgk4pyuj203ikrtgpthyuj},
and~\eqref{o1edRsjdepwrgk4pyuj203ikrtgpthyuj2}, it follows that
\begin{equation*}\begin{split}&
\frac1{2^{n+1+s}}\int_{{\mathcal{S}}\cap (B_1\setminus B_r)}\frac{ u(y)}{|y|^{n+1+s}}\,d{\mathcal{H}}^n_y -\frac{C_0\,(u(\zeta)+u(x_\star))}{\delta^{1+s}r^{1+s}}\\&\quad\le
{\mathcal{L}}_{B_1} (u-P_\lambda)(x_\star)+
{\mathcal{L}}_{B_1} P_\lambda(x_\star)\\&\quad
={\mathcal{L}}_{B_1} u(x_\star)\\&\quad\le M \,u(x_\star) .\end{split}
\end{equation*}
We thereby conclude that
\begin{equation*}
\frac1{2^{n+1+s}}\int_{{\mathcal{S}}\cap (B_1\setminus B_r)}\frac{ u(y)}{|y|^{n+1+s}}\,d{\mathcal{H}}^n_y \le\frac{ (M\delta^{1+s} r^{1+s}+C_0) \,(u(\zeta)+u(x_\star))}{\delta^{1+s} r^{1+s}}.
\end{equation*}

Notice also that
$$ u(x_\star)=P_\lambda(x_\star)\le P_\lambda(\zeta)\le u(\zeta),$$
leading to
\begin{equation*}
\int_{{\mathcal{S}}\cap (B_1\setminus B_r)}\frac{ u(y)}{|y|^{n+1+s}}\,d{\mathcal{H}}^n_y \le\frac{2^{n+2+s}(M\delta^{1+s} r^{1+s}+C_0) \,u(\zeta)}{\delta^{1+s} r^{1+s}}.
\end{equation*}
Hence, integrating over~$\zeta\in {\mathcal{S}}\cap B_{\delta r/10}(z_r)$,
\begin{equation*}
\int_{{\mathcal{S}}\cap (B_1\setminus B_r)}\frac{ u(y)}{|y|^{n+1+s}}\,d{\mathcal{H}}^n_y \le\frac{C_\star(M\delta^{1+s} r^{1+s}+C_0)}{\delta^{n+1+s} r^{n+1+s}}\int_{{\mathcal{S}}\cap B_{\delta r/10}(z_r)} u(\zeta)\,d{\mathcal{H}}^n_\zeta,
\end{equation*}which, up to renaming~$C_\star$, gives the desired result,
since~${\mathcal{S}}\cap B_{\delta r/10}(z_r)\subseteq
(B_{r}\setminus B_{r/2})\cap\{u>0\}$.
\end{proof}

The estimate in~\eqref{oq2jfgmoptyhlp32ktgkjogtn4839n19vib6iyoy} is useful
in itself, since it implies that
weighted averages on annuli of radius of order~$r$
grow at least as a negative power of~$r$. More precisely, we have that:

\begin{lemma}\label{CONS1}
Assume that there exist~$r_\star\in(0,1)$ and~$C_\star>0$ such that~\eqref{oq2jfgmoptyhlp32ktgkjogtn4839n19vib6iyoy} holds true for all~$r\in(0,r_\star]$.

Then, there exist~$c$, $\beta>0$, depending only on~$C_\star$, such that,
for all~$r\in(0,r_\star]$,
$$ \int_{{\mathcal{S}}\cap (B_{r}\setminus B_{r/2})} \frac{ u(y)}{|y|^{n+1+s}}\,d{\mathcal{H}}^n_y\ge
\frac{c\,r_\star^\beta}{r^\beta}\,\int_{{\mathcal{S}}\cap (B_{1}\setminus B_{1/2})} \frac{ u(y)}{|y|^{n+1+s}}\,d{\mathcal{H}}^n_y.$$
\end{lemma}

\begin{proof} Let~$\omega\in\left[\frac12,1\right)$, to be conveniently chosen below.
We define~$r_k:= \frac{\omega\,r_\star}{2^k}$ and
$$ a_k:= \int_{{\mathcal{S}}\cap (B_{r_k}\setminus B_{r_k/2})} \frac{ u(y)}{|y|^{n+1+s}}\,d{\mathcal{H}}^n_y.$$
We remark that
\begin{eqnarray*}&&
\sum_{i=0}^{k-1}a_i
=\sum_{i=0}^{k-1}
\int_{{\mathcal{S}}\cap (B_{\omega r_\star/2^{i}}\setminus B_{\omega r_\star/2^{i+1}})} \frac{ u(y)}{|y|^{n+1+s}}\,d{\mathcal{H}}^n_y\\&&\quad=\int_{{\mathcal{S}}\cap (B_{\omega r_\star}\setminus B_{\omega r_\star/2^{k}})} \frac{ u(y)}{|y|^{n+1+s}}\,d{\mathcal{H}}^n_y=\int_{{\mathcal{S}}\cap (B_{\omega r_\star}\setminus B_{r_k})} \frac{ u(y)}{|y|^{n+1+s}}\,d{\mathcal{H}}^n_y
\end{eqnarray*}
and therefore, by~\eqref{oq2jfgmoptyhlp32ktgkjogtn4839n19vib6iyoy}
(used here with~$r:=r_k$),
\begin{equation}\label{3fon2lyq1l2jer24o35nlye1234nv} \sum_{i=0}^{k-1}a_i\le
 \int_{{\mathcal{S}}\cap (B_1\setminus B_{r_k})}\frac{ u(y)}{|y|^{n+1+s}}\,d{\mathcal{H}}^n_y\le
C_\star \int_{{\mathcal{S}}\cap (B_{r_k}\setminus B_{r_k/2})} \frac{ u(y)}{|y|^{n+1+s}}\,d{\mathcal{H}}^n_y
= C_\star \,a_k.\end{equation}
As a result,
$$  (C_\star+1)\sum_{i=0}^{k-1}a_i\le C_\star \sum_{i=0}^{k-1}a_i
+C_\star \,a_k= C_\star\sum_{i=0}^{k}a_i.$$

Thus, setting
$$ S_k:=\sum_{i=0}^{k}a_i\qquad{\mbox{and}}\qquad
C:=\frac{C_\star+1}{C_\star}\in(1,+\infty),$$we have that~$S_k\ge C S_{k-1}$, leading to~$ S_k\ge C^{k-1}\,S_1$.

Then, defining~$\beta:=\frac{\ln C}{\ln2}$ and using again~\eqref{3fon2lyq1l2jer24o35nlye1234nv},
\begin{eqnarray*}&&
\frac{2^{k\beta}}{C^2\,(C_\star+1)\,\omega^\beta}
\int_{{\mathcal{S}}\cap (B_{\omega/2}\setminus B_{\omega r_\star/4})} \frac{ u(y)}{|y|^{n+1+s}}\,d{\mathcal{H}}^n_y\\&&\quad
\le
\frac{2^{k\beta} C^{\frac{\ln\omega}{\ln2}}}{C\,(C_\star+1)\,\omega^\beta}
\int_{{\mathcal{S}}\cap (B_{\omega r_\star/2}\setminus B_{\omega r_\star/4})} \frac{ u(y)}{|y|^{n+1+s}}\,d{\mathcal{H}}^n_y\\&&\quad
=\frac{2^{k\beta} C^{\frac{\ln\omega}{\ln2}}a_1}{C\,(C_\star+1)\,\omega^\beta} \le
\frac{2^{k\beta}C^{\frac{\ln\omega}{\ln2}}S_1}{C\,(C_\star+1)\,\omega^\beta}=
\frac{C^{k-1}\,S_1}{C_\star+1}
\le\frac{S_k}{C_\star+1}\\&&\quad
=\frac{1}{C_\star+1}\left( \sum_{i=0}^{k-1}a_i+a_k\right)\le a_k=\int_{{\mathcal{S}}\cap (B_{\omega r_\star/2^k}\setminus B_{\omega r_\star/2^{k+1}})} \frac{ u(y)}{|y|^{n+1+s}}\,d{\mathcal{H}}^n_y.\end{eqnarray*}

Hence, picking~$k\in\N$ such that~$\frac{r_\star}{2^{k+1}}\le r<\frac{r_\star}{2^k}$ and taking~$\omega:=\frac{2^kr}{{r_\star}}\in\left[\frac12,1\right)$, we find that
\begin{equation}\label{oq2jfgmoptyhlp32ktgkjogtn4839n19vib6iyoy77} \frac{{r_\star^\beta}}{C^2\,(C_\star+1)\,r^\beta}
\int_{{\mathcal{S}}\cap (B_{\omega{r_\star}/2}\setminus B_{\omega{r_\star}/4})} \frac{ u(y)}{|y|^{n+1+s}}\,d{\mathcal{H}}^n_y\le
\int_{{\mathcal{S}}\cap (B_{r}\setminus B_{r/2})} \frac{ u(y)}{|y|^{n+1+s}}\,d{\mathcal{H}}^n_y.\end{equation}

Also, by~\eqref{oq2jfgmoptyhlp32ktgkjogtn4839n19vib6iyoy} (used here with~$r:=\frac{\omega{r_\star}}2$),
$$ \int_{{\mathcal{S}}\cap (B_1\setminus B_{\omega{r_\star}/2})}\frac{ u(y)}{|y|^{n+1+s}}\,d{\mathcal{H}}^n_y \le
C_\star \int_{{\mathcal{S}}\cap (B_{\omega{r_\star}/2}\setminus B_{\omega{r_\star}/4})} \frac{ u(y)}{|y|^{n+1+s}}\,d{\mathcal{H}}^n_y.$$Combining this with~\eqref{oq2jfgmoptyhlp32ktgkjogtn4839n19vib6iyoy77}, we find that
$$ \frac{{r_\star^\beta}}{C^2\,C_\star\,(C_\star+1)\,r^\beta}
\int_{{\mathcal{S}}\cap (B_{1}\setminus B_{\omega{r_\star}/2})} \frac{ u(y)}{|y|^{n+1+s}}\,d{\mathcal{H}}^n_y\le
\int_{{\mathcal{S}}\cap (B_{r}\setminus B_{r/2})} \frac{ u(y)}{|y|^{n+1+s}}\,d{\mathcal{H}}^n_y.$$
Since~$\frac{\omega{r_\star}}2<\frac12$, and thus~$B_{\omega{r_\star}/2}\subseteq B_{1/2}$,
the inequality above
leads to the desired result.
\end{proof}

\subsection{Another geometric boundary Harnack inequality}

Now we use the annular growth detected in Lemma~\ref{CONS1} (as a consequence of Lemma~\ref{PplemdfltX34tyh.2})
to obtain
another form of the geometric boundary Harnack inequality, which is related in spirit to,
but technically different from, the one provided by Lemma~\ref{LE:31}.

To accomplish this plan, a useful step is supplied by the following simple
observation:

\begin{lemma}\label{SLKJmdwlegb.o3grlhjnm45tg}
Let~$\alpha\in(s,1]$ and~${\mathcal{S}}\subseteq\R^{n+1}$ be a bounded, connected hypersurface of class~$C^{1,\alpha}$
such that~$0\in{\mathcal{S}}$. 

For any~$r\in(0,1)$, let~$f_r(x):=f(rx)$ and~${\mathcal{S}}_r:=\frac{{\mathcal{S}}}r$. Then, for all~$x\in
{\mathcal{S}}_r\cap B_{1/2}$,
\begin{equation*}\begin{split}&
\int_{{\mathcal{S}}_r\cap B_1}\frac{ f_r(y)-f_r(x)}{|x-y|^{n+1+s}}\,d{\mathcal{H}}^n_y
=r^{1+s}\Bigg[\int_{{\mathcal{S}}\cap B_1}\frac{ f(y)-f(rx)}{|rx-y|^{n+1+s}}\,d{\mathcal{H}}^n_y\\&\qquad\qquad
-\int_{{\mathcal{S}}\cap (B_1\setminus B_r)}\frac{ f(y)}{|rx-y|^{n+1+s}}\,d{\mathcal{H}}^n_y
+f(rx)
\int_{{\mathcal{S}}\cap (B_1\setminus B_r)}\frac{d{\mathcal{H}}^n_y }{|rx-y|^{n+1+s}}\Bigg].\end{split}
\end{equation*}
\end{lemma}

\begin{proof} We use the notation~$X:=rx$ and~$Y:=rY$ and we see that
\begin{equation*}\begin{split}\int_{{\mathcal{S}}_r\cap B_1}\frac{ f_r(y)-f_r(x)}{|x-y|^{n+1+s}}\,d{\mathcal{H}}^n_y&=
\int_{{\mathcal{S}}_r\cap B_1}\frac{ f(ry)-f(rx)}{|x-y|^{n+1+s}}\,d{\mathcal{H}}^n_y\\&=
r^{1+s}\int_{{\mathcal{S}}\cap B_r}\frac{ f(Y)-f(X)}{|X-Y|^{n+1+s}}\,d{\mathcal{H}}^n_Y.\end{split}
\end{equation*}
This leads to the desired result.
\end{proof}

Then we have a geometric boundary Harnack inequality, with a statement prone to be iterated:

\begin{lemma}\label{LEM5.1}
Let~$\alpha\in(s,1]$ and~${\mathcal{S}}\subseteq\R^{n+1}$ be a bounded, connected hypersurface of class~$C^{1,\alpha}$
such that~$0\in{\mathcal{S}}$. 

Given a set~$V\subseteq\R^{n+1}$, let
$${\mathcal{L}}_{{\mathcal{S}},V} f(x):=\int_{{\mathcal{S}}\cap V}\frac{ f(y)-f(x)}{|x-y|^{n+1+s}}\,d{\mathcal{H}}^n_y.$$

Let~$u$, $\overline v:{\mathcal{S}}\to\R$
be such that\begin{equation}\label{EAFYJLDIKSPpwk3mbm01LK-vvNMASP.3}
{\mbox{$0\le \overline v\le u$ in~${\mathcal{S}}\cap B_1$.}}\end{equation}

Assume that, in~${\mathcal{S}}\cap B_{1/2}\cap\{u>0\}$,
\begin{equation}\label{LAxIKNSDGTBSDLS0-23rfkLSLD}\begin{split}&
{\mathcal{L}}_{{\mathcal{S}},B_1} u\le Mu\\
{\mbox{and }}\quad&
|{\mathcal{L}}_{{\mathcal{S}},B_1} u(x)|+|{\mathcal{L}}_{{\mathcal{S}},B_1} \overline v(x)|\le M\int_{{\mathcal{S}}\cap (B_1\setminus B_{1/2})}u(y)\,d{\mathcal{H}}^n_y,\end{split}\end{equation}
for some~$M>0$.

Assume additionally that there exists~$\delta\in\left(0,\frac1{10}\right)$ such that,
for every~$r\in\left(0,\frac12\right)$,
\begin{equation*}
{\mbox{there exists~$z_r\in{\mathcal{S}}\cap (B_{9r/10}\setminus B_{7r/10})$
such that~${\mathcal{S}}\cap B_{\delta r}(z_r)\subseteq\{u>0\}$.}}\end{equation*}

Then, there exist~$\eta$, 
$\rho\in\left(0,\frac1{100}\right)$,
depending only on~$n$, $s$, $\alpha$, the $C^{1,\alpha}$-regularity of~${{\mathcal{S}}}$, $\delta$, $C$,
and~$M$, such that, in~${\mathcal{S}}\cap B_{\rho}$,
\begin{equation}\label{EAFYJLDIKSPpwk3mbm01LK-vvNMASP}
{\mbox{either~$\eta u\le \overline v$ (case~1), or~$\overline v\le(1-\eta)u$ (case~2).}}\end{equation}

Moreover, setting~${\mathcal{S}}_\rho:=\frac{{\mathcal{S}}}\rho$ and,
for all~$x\in{\mathcal{S}}_\rho\cap B_1$,
\begin{equation*}\begin{split}& {\mbox{$u_\rho(x):=u(\rho x)$ and}}\\
&{\mbox{either
$v_\rho(x):=\displaystyle\frac{\overline v(\rho x)-\eta u(\rho x)}{1-\eta}$ (case 1), or
$v_\rho(x):=\displaystyle\frac{\overline v(\rho x)}{1-\eta}$ (case 2),}}
\end{split}\end{equation*}
we have that
\begin{equation}\label{EAFYJLDIKSPpwk3mbm01LK-vvNMASP.2}
{\mbox{$0\le v_\rho\le u_\rho$ in~${\mathcal{S}}_\rho\cap B_1$}}\end{equation}
and, in~${\mathcal{S}}_\rho\cap B_{1/2}\cap\{u_\rho>0\}$,
\begin{equation}\label{smwq908vr9043 85219-BGEAV:L-9143mvn5}
\begin{split}&
{\mathcal{L}}_{{\mathcal{S}}_\rho,B_1} u_\rho\le \widehat{C}\,u_\rho\\
{\mbox{and }}\quad&
|{\mathcal{L}}_{{\mathcal{S}}_\rho,B_1} u_\rho|+|{\mathcal{L}}_{{\mathcal{S}}_\rho,B_1} v_\rho|\le \widehat{C}\int_{{\mathcal{S}}_\rho\cap (B_1\setminus B_{1/2})}u_\rho(y)\,d{\mathcal{H}}^n_y,\end{split}\end{equation}
for a suitable~$\widehat{C}>0$
depending only on~$n$, $s$, $\alpha$, the $C^{1,\alpha}$-regularity of~${{\mathcal{S}}}$, $\delta$, and~$C$
(but independent of~$M$).
\end{lemma}

\begin{proof} We use the short notation
$$ I:=\int_{{\mathcal{S}}\cap (B_1\setminus B_{1/2})}u(y)\,d{\mathcal{H}}^n_y.$$
In view of Lemma~\ref{PplemdfltX34tyh.2}, 
the estimate in~\eqref{oq2jfgmoptyhlp32ktgkjogtn4839n19vib6iyoy} holds true, and this allows us to
use Lemma~\ref{CONS1} and conclude (up to renaming constants) that
\begin{equation}\label{AOSJNbSs203tjghnt90o3rtg1qwsxfty678uhjo0wo1-2rjgb}
\int_{{\mathcal{S}}\cap (B_{\rho}\setminus B_{\rho/2})} u(y)\,d{\mathcal{H}}^n_y\ge
c\,{r_\star^\beta}\,I\,\rho^{n+1+s-\beta}.\end{equation}

Moreover, we recall that,
by~\cite[Lemma~6.1]{MAXPLE}, for all~$r\in(0,1)$, and\footnote{Strictly speaking, in~\cite{MAXPLE}
the ambient hypersurface was assumed to be of class~$C^2$, but in fact it suffices, with the same proof, to assume~$C^{1,\alpha}$ with~$\alpha\in(s,1]$. Also, after scaling,
the estimate there would produce~$\int_{{\mathcal{S}}\cap B_{r}}u(y)\,d{\mathcal{H}}^n_y$ on the right-hand side
of~\eqref{olsmdcjp45oyhljXy}: however, the same proof produces the sharper bound~$\int_{{\mathcal{S}}\cap (B_{r}\setminus B_{r/2})}u(y)\,d{\mathcal{H}}^n_y$. A similar observation was put forth in the remark after Lemma~5.2 in~\cite{MR3661864}.} up to renaming~$C_\star$ with respect to~\eqref{oq2jfgmoptyhlp32ktgkjogtn4839n19vib6iyoy},
\begin{equation}\label{olsmdcjp45oyhljXy} \sup_{{\mathcal{S}}\cap B_{r/2}}u\le \frac{C_\star}{r^n}\,\int_{{\mathcal{S}}\cap (B_{r}\setminus B_{r/2})}u(y)\,d{\mathcal{H}}^n_y.\end{equation}

To prove the desired result, 
we pick~$\rho>0$, to be taken suitably small here below, and distinguish two cases: either (case~1)
$$ \int_{{\mathcal{S}}\cap (B_\rho\setminus B_{\rho/2})}\overline v(y)\,d{\mathcal{H}}^n_y\ge \frac12\int_{{\mathcal{S}}\cap (B_\rho\setminus B_{\rho/2})}u(y)\,d{\mathcal{H}}^n_y,$$
or (case~2)
$$ \int_{{\mathcal{S}}\cap (B_\rho\setminus B_{\rho/2})}\overline v(y)\,d{\mathcal{H}}^n_y<\frac{1}2
\int_{{\mathcal{S}}\cap (B_\rho\setminus B_{\rho/2})}u(y)\,d{\mathcal{H}}^n_y.$$
In case~1, we define~$w:=\overline v-\eta u$, with~$\eta\in\left(0,\frac1{100}\right)$ to be taken conveniently small. We see that
\begin{equation}\label{AOSJNbSs203tjghnt90o3rtg1qwsxfty678uhjo0wo1-2rjgb1}\begin{split}
\int_{{\mathcal{S}}\cap (B_\rho\setminus B_{\rho/2})}w(y)\,d{\mathcal{H}}^n_y&
\ge \left(\frac12-\eta\right) \int_{{\mathcal{S}}\cap (B_\rho\setminus B_{\rho/2})}u(y)\,d{\mathcal{H}}^n_y\\&\ge\frac{1}4
\int_{{\mathcal{S}}\cap (B_\rho\setminus B_{\rho/2})}u(y)\,d{\mathcal{H}}^n_y.\end{split}\end{equation}
Also, since~$\overline v\ge0$, we find that\begin{equation}\label{olsmdcjp45oyhljXy.00}w\ge-\eta u.\end{equation}

In case~2, we define~$w:=(1-\eta) u-\overline v$. We see that
\begin{equation}\label{AOSJNbSs203tjghnt90o3rtg1qwsxfty678uhjo0wo1-2rjgb0} \begin{split}\int_{{\mathcal{S}}\cap (B_\rho\setminus B_{\rho/2})}w(y)\,d{\mathcal{H}}^n_y&
\ge \left(\frac12-\eta\right) \int_{{\mathcal{S}}\cap (B_\rho\setminus B_{\rho/2})}u(y)\,d{\mathcal{H}}^n_y\\&\ge\frac{1}4
\int_{{\mathcal{S}}\cap (B_\rho\setminus B_{\rho/2})}u(y)\,d{\mathcal{H}}^n_y.\end{split}\end{equation}
Also, since~$\overline v\le u$, we find that\begin{equation}\label{olsmdcjp45oyhljXy.0}w\ge(1-\eta) u-u=-\eta u.\end{equation}

Interestingly, combining~\eqref{olsmdcjp45oyhljXy} (here with~$r:=1$) and either~\eqref{olsmdcjp45oyhljXy.00}
or~\eqref{olsmdcjp45oyhljXy.0}, we find that, in~${\mathcal{S}}\cap B_{1/2}$,
\begin{equation}\label{ojwqnd0436y546jonkmo24iv5mb60p7un203cn9o3}
w\ge-C_\star\,\eta\,I.\end{equation}
Moreover, using~\eqref{AOSJNbSs203tjghnt90o3rtg1qwsxfty678uhjo0wo1-2rjgb}
and either~\eqref{AOSJNbSs203tjghnt90o3rtg1qwsxfty678uhjo0wo1-2rjgb1}
or~\eqref{AOSJNbSs203tjghnt90o3rtg1qwsxfty678uhjo0wo1-2rjgb0}, we see that
\begin{equation}\label{92ikfnv39enYUNnt90o3rtg1qwsxfty678uhjo0wo1-2rjgb1}
\int_{{\mathcal{S}}\cap (B_\rho\setminus B_{\rho/2})}w(y)\,d{\mathcal{H}}^n_y
\ge\frac{c\,{r_\star^\beta}\,I\,\rho^{n+1+s-\beta}}{4}.
\end{equation}
Furthermore, by~\eqref{LAxIKNSDGTBSDLS0-23rfkLSLD}, for all~$x\in{\mathcal{S}}\cap B_{1/2}$,
\begin{equation}\label{LAxIKNSDGTBSDLS0-23rfkLSLD.01iprk4gmNSl}
|{\mathcal{L}}_{{\mathcal{S}},B_1}w(x)|\le 2M\int_{{\mathcal{S}}\cap (B_1\setminus B_{1/2})}u(y)\,d{\mathcal{H}}^n_y=2M\,I.\end{equation}

Notice also that the desired result in~\eqref{EAFYJLDIKSPpwk3mbm01LK-vvNMASP} is proved if we show that~$w\ge0$ in~${\mathcal{S}}\cap B_{\rho}$.
Hence, we argue by contradiction, supposing that there exists~$x_0\in{\mathcal{S}}\cap B_{\rho}$ such that~$w(x_0)<0$.

Then, given~$\lambda\in\R$, we consider the parabola
$$ P_\lambda(x):=-\frac{16\,C_\star\,\eta\,I}{\rho^2}|x-x_0|^2-\lambda.$$
When~$\lambda\ge C_\star\,\eta\,I$,
we have that~$P_\lambda\le -\lambda\le-C_\star\,\eta\,I\le w$
in~${\mathcal{S}}\cap B_{1/2}$, thanks to~\eqref{ojwqnd0436y546jonkmo24iv5mb60p7un203cn9o3}.
Also, if~$\lambda=0$, we have that~$P_\lambda(x_0)=0>w(x_0)$.
As a result, we find~$\lambda\in[0,C_\star\,\eta\,I]$ such that~$P_\lambda$ touches~$w$
in~${\mathcal{S}}\cap B_{1/2}$ from below at some point~$y_0$.

We also observe that, by~\eqref{ojwqnd0436y546jonkmo24iv5mb60p7un203cn9o3},
\begin{equation*} -C_\star\,\eta\,I\le w(y_0)=P_\lambda(y_0)=-\frac{16\,C_\star\,\eta\,I}{\rho^2}|y_0-x_0|^2-\lambda\le
-\frac{16\,C_\star\,\eta\,I}{\rho^2}|y_0-x_0|^2
\end{equation*}
and therefore
\begin{equation}\label{NHAEGD0-94vc238mJKbX}y_0\in B_{\rho/4}(x_0)\subseteq B_{2\rho}.\end{equation}

It is also helpful to observe that, by either~\eqref{olsmdcjp45oyhljXy.00} or~\eqref{olsmdcjp45oyhljXy.0},
\begin{eqnarray*}&&
\int_{{\mathcal{S}}\cap (B_1\setminus B_{1/2})}\frac{ w(y)-P_\lambda(y_0)}{|y_0-y|^{n+1+s}}\,d{\mathcal{H}}^n_y\ge
\int_{{\mathcal{S}}\cap (B_1\setminus B_{1/2})}\frac{ w(y)}{|y_0-y|^{n+1+s}}\,d{\mathcal{H}}^n_y\\&&\qquad\ge-\eta
\int_{{\mathcal{S}}\cap (B_1\setminus B_{1/2})}\frac{u(y)}{|y_0-y|^{n+1+s}}\,d{\mathcal{H}}^n_y\ge -\underline{C}\,\eta\,I,
\end{eqnarray*}for some~$\underline C>0$, depending only on~$n$, $s$, and~${{\mathcal{S}}}$.

Thus, in light of~\eqref{ojwqnd0436y546jonkmo24iv5mb60p7un203cn9o3},
\eqref{LAxIKNSDGTBSDLS0-23rfkLSLD.01iprk4gmNSl}, and~\eqref{NHAEGD0-94vc238mJKbX},
\begin{equation}\label{JOLASbjA.MSv6n78DPK-k3rf:0023edfvm.1}
\begin{split}
2MI&\ge{\mathcal{L}}_{{\mathcal{S}},B_1}w(y_0)\\&=
\int_{{\mathcal{S}}\cap B_1}\frac{ w(y)-P_\lambda(y_0)}{|y_0-y|^{n+1+s}}\,d{\mathcal{H}}^n_y\\&\ge
\int_{{\mathcal{S}}\cap B_{4\rho}(y_0)}\frac{ w(y)-P_\lambda(y_0)}{|y_0-y|^{n+1+s}}\,d{\mathcal{H}}^n_y\\&\qquad\qquad+
\int_{{\mathcal{S}}\cap (B_{1/2}\setminus B_{4\rho}(y_0))}\frac{ 
-C_\star\,\eta\,I+0}{|y_0-y|^{n+1+s}}\,d{\mathcal{H}}^n_y -\underline{C}\,\eta\,I\\&\ge
\int_{{\mathcal{S}}\cap B_{4\rho}(y_0)}\frac{ w(y)-P_\lambda(y)}{|y_0-y|^{n+1+s}}\,d{\mathcal{H}}^n_y\\&\qquad\qquad+
\int_{{\mathcal{S}}\cap B_{4\rho}(y_0)}\frac{ P_\lambda(y)-P_\lambda(y_0)}{|y_0-y|^{n+1+s}}\,d{\mathcal{H}}^n_y-
\frac{\widehat C\,C_\star\eta\,I}{\rho^{1+s}} -\underline{C}\,\eta\,I,
\end{split}\end{equation}
for some~$\widehat C>0$, depending only on~$n$, $s$, and~${{\mathcal{S}}}$.

We also recall that, as in~\eqref{qowjhdfnt034rgrbRF.3rtgrbfRb0pjmfewvoUhnhyshnd},
\begin{equation*}\left|
\int_{{\mathcal{S}}\cap B_{4\rho}(y_0)}\frac{ P_\lambda(y)-P_\lambda(y_0)}{|y_0-y|^{n+1+s}}\,d{\mathcal{H}}^n_y
\right|\le
\widehat C\,\rho^{1-s}\,
[P_\lambda]_{C^{1,1}( B_{4\rho}(y_0))}\le
\frac{\widehat C\,C_\star\,\eta\,I}{\rho^{1+s}},
\end{equation*}
up to freely renaming~$\widehat C>0$.

Combining this and~\eqref{JOLASbjA.MSv6n78DPK-k3rf:0023edfvm.1}, and observing that~$B_{4\rho}(y_0)\supseteq B _\rho$, up to renaming constants we infer that
\begin{equation}\label{ojsdn499rg.092rojrtg34gbm2vn4qec1qb6ftHHX2}\begin{split}
2MI+\frac{\widehat C\,C_\star\,\eta\,I}{\rho^{1+s}}+\underline{C}\,\eta\,I&\ge
\int_{{\mathcal{S}}\cap B_{4\rho}(y_0)}\frac{ w(y)-P_\lambda(y)}{|y_0-y|^{n+1+s}}\,d{\mathcal{H}}^n_y\\&
\ge
\int_{{\mathcal{S}}\cap (B_{\rho}\setminus B_{\rho/2})}\frac{ w(y)-P_\lambda(y)}{|y_0-y|^{n+1+s}}\,d{\mathcal{H}}^n_y.\end{split}
\end{equation}

Also, if~$y\in{\mathcal{S}}\cap (B_{\rho}\setminus B_{\rho/2})$,
$$ |y_0-y|\le|y_0|+|y|\le2\rho+|y|\le 5|y|.$$
Hence we deduce from~\eqref{ojsdn499rg.092rojrtg34gbm2vn4qec1qb6ftHHX2} that, up to renaming constants,
\begin{equation*}\begin{split}
\left(M+\frac{\widehat C\,C_\star\,\eta}{\rho^{1+s}}+\underline{C}\,\eta\right)I&\ge
\int_{{\mathcal{S}}\cap (B_{\rho}\setminus B_{\rho/2})}\frac{ w(y)-P_\lambda(y)}{|y|^{n+1+s}}\,d{\mathcal{H}}^n_y\\&\ge
\int_{{\mathcal{S}}\cap (B_{\rho}\setminus B_{\rho/2})}\frac{ w(y)}{|y|^{n+1+s}}\,d{\mathcal{H}}^n_y.\end{split}
\end{equation*}
This and~\eqref{92ikfnv39enYUNnt90o3rtg1qwsxfty678uhjo0wo1-2rjgb1} yield that, up to renaming constants,
\begin{equation*}
M+\frac{\widehat C\,C_\star\,\eta}{\rho^{1+s}}+\underline{C}\,\eta\ge\frac{c\,{r_\star^\beta}}{\rho^{\beta}}.
\end{equation*}
Hence, if~$\rho:=\left(\frac{c}{2M}\right)^{\frac1\beta}{r_\star}$, we conclude that
$$\frac{\widehat C\,C_\star\,\eta}{\rho^{1+s}}+\underline C\,\eta\ge M.$$
A contradiction now follows by choosing~$\eta$ suitably small
(possibly also in dependence of~$\rho$). This establishes the claim in~\eqref{EAFYJLDIKSPpwk3mbm01LK-vvNMASP}, as desired.

Also, the claim in~\eqref{EAFYJLDIKSPpwk3mbm01LK-vvNMASP.2} follows from~\eqref{EAFYJLDIKSPpwk3mbm01LK-vvNMASP.3}
and~\eqref{EAFYJLDIKSPpwk3mbm01LK-vvNMASP}.

We now deal with the proof of~\eqref{smwq908vr9043 85219-BGEAV:L-9143mvn5}. First, we point out that, by Lemma~\ref{SLKJmdwlegb.o3grlhjnm45tg}, for all~$x\in{\mathcal{S}}_\rho\cap B_{1/2}\cap\{u_\rho>0\}$,
\begin{equation}\label{09mcBUSLSi024rv83S59876}
\begin{split}&
{\mathcal{L}}_{{\mathcal{S}}_\rho,B_1} u_\rho(x)=
\int_{{\mathcal{S}}_\rho\cap B_1}\frac{ u_\rho(y)-u_\rho(x)}{|x-y|^{n+1+s}}\,d{\mathcal{H}}^n_y\\&\;
=\rho ^{1+s}\Bigg[\int_{{\mathcal{S}}\cap B_1}\frac{ u(y)-u(\rho x)}{|\rho x-y|^{n+1+s}}\,d{\mathcal{H}}^n_y
-\int_{{\mathcal{S}}\cap (B_1\setminus B_\rho)}\frac{ u(y)}{|\rho x-y|^{n+1+s}}\,d{\mathcal{H}}^n_y\\&\qquad\qquad
+u(\rho x)
\int_{{\mathcal{S}}\cap (B_1\setminus B_\rho)}\frac{d{\mathcal{H}}^n_y }{|\rho x-y|^{n+1+s}}\Bigg]\\&\;
=\rho ^{1+s}\Bigg[{\mathcal{L}}_{{\mathcal{S}},B_1} u(\rho x)
-\int_{{\mathcal{S}}\cap (B_1\setminus B_\rho)}\frac{ u(y)}{|\rho x-y|^{n+1+s}}\,d{\mathcal{H}}^n_y\\&\qquad\qquad
+u(\rho x)
\int_{{\mathcal{S}}\cap (B_1\setminus B_\rho)}\frac{d{\mathcal{H}}^n_y }{|\rho x-y|^{n+1+s}}\Bigg]
\end{split}
\end{equation}
and therefore
$${\mathcal{L}}_{{\mathcal{S}}_\rho,B_1} u_\rho(x)\le\rho ^{1+s}\Bigg[{\mathcal{L}}_{{\mathcal{S}},B_1} u(\rho x)
+0+u(\rho x)
\int_{{\mathcal{S}}\cap (B_1\setminus B_\rho)}\frac{d{\mathcal{H}}^n_y }{|\rho x-y|^{n+1+s}}\Bigg].$$

We also notice that
\begin{equation}\label{VATIKE094c1043NMmb098L7b8}
\rho^{1+s} 
\int_{{\mathcal{S}}\cap (B_1\setminus B_\rho)}\frac{d{\mathcal{H}}^n_y }{|\rho x-y|^{n+1+s}}
\le2^{n+1+s}\rho^{1+s} 
\int_{{\mathcal{S}}\cap (B_1\setminus B_\rho)}\frac{d{\mathcal{H}}^n_y }{|y|^{n+1+s}}\le\frac{\widehat{C}}2,\end{equation}
for a suitable choice of constants.

Hence, recalling~\eqref{LAxIKNSDGTBSDLS0-23rfkLSLD},
\begin{eqnarray*}&&
{\mathcal{L}}_{{\mathcal{S}}_\rho,B_1} u_\rho(x)\le\left(
\rho^{1+s} M+\frac{\widehat{C}}2\right)u_\rho( x),
\end{eqnarray*}
and the first claim in~\eqref{smwq908vr9043 85219-BGEAV:L-9143mvn5} follows
(up to relabeling constants).

Besides, we deduce from~\eqref{09mcBUSLSi024rv83S59876} that
\begin{equation}\label{0mec32VYPSv-4360mcv520}
\begin{split}&|
{\mathcal{L}}_{{\mathcal{S}}_\rho,B_1} u_\rho(x)|\\&\;\le
\rho ^{1+s}\Bigg[|{\mathcal{L}}_{{\mathcal{S}},B_1} u(\rho x)|
+\int_{{\mathcal{S}}\cap (B_1\setminus B_\rho)}\frac{ u(y)}{|\rho x-y|^{n+1+s}}\,d{\mathcal{H}}^n_y\\&\qquad\qquad
+u(\rho x)
\int_{{\mathcal{S}}\cap (B_1\setminus B_\rho)}\frac{d{\mathcal{H}}^n_y }{|\rho x-y|^{n+1+s}}\Bigg].
\end{split}
\end{equation}

In addition, by~\eqref{LAxIKNSDGTBSDLS0-23rfkLSLD} and Lemma~\ref{CONS1},
up to renaming constants,
\begin{equation}\label{0mec32VYPSv-4360mcv521}
\begin{split}
\rho ^{1+s}\Big(|{\mathcal{L}}_{{\mathcal{S}},B_1} u(\rho x)|
+|{\mathcal{L}}_{{\mathcal{S}},B_1} \overline v(\rho x)|\Big)
&\le\rho ^{1+s}
M\int_{{\mathcal{S}}\cap (B_1\setminus B_{1/2})}u(y)\,d{\mathcal{H}}^n_y\\&\le\frac{\rho ^{1+s+\beta}
M}{c\,r_\star^\beta}
\int_{{\mathcal{S}}\cap (B_{\rho}\setminus B_{\rho/2})} \frac{ u(y)}{|y|^{n+1+s}}\,d{\mathcal{H}}^n_y\\&
\le\frac{\widehat{C}\rho ^{\beta-n}M}{r_\star^\beta}
\int_{{\mathcal{S}}\cap (B_{\rho}\setminus B_{\rho/2})} u(y)\,d{\mathcal{H}}^n_y\\&=\frac{\widehat{C}\rho ^{\beta}M}{r_\star^\beta}
\int_{{\mathcal{S}}_\rho\cap (B_1\setminus B_{1/2})} u_\rho(y)\,d{\mathcal{H}}^n_y.
\end{split}
\end{equation}

Furthermore, recalling Lemma~\ref{PplemdfltX34tyh.2} and renaming constants when necessary,
\begin{equation}\label{0mec32VYPSv-4360mcv522}
\begin{split}&\rho ^{1+s}
\int_{{\mathcal{S}}\cap (B_1\setminus B_\rho)}\frac{ u(y)}{|\rho x-y|^{n+1+s}}\,d{\mathcal{H}}^n_y\le2^{n+1+s}\rho ^{1+s}
\int_{{\mathcal{S}}\cap (B_1\setminus B_\rho)}\frac{ u(y)}{|y|^{n+1+s}}\,d{\mathcal{H}}^n_y\\&\quad\le
C_\star \rho ^{1+s}\int_{{\mathcal{S}}\cap (B_{\rho}\setminus B_{\rho/2})} \frac{ u(y)}{|y|^{n+1+s}}\,d{\mathcal{H}}^n_y\le
\frac{C_\star}{ \rho ^n}\int_{{\mathcal{S}}\cap (B_{\rho}\setminus B_{\rho/2})} u(y)\,d{\mathcal{H}}^n_y\\&\quad=
{C_\star}\int_{{\mathcal{S}}_\rho\cap (B_{1}\setminus B_{1/2})} u_\rho(y)\,d{\mathcal{H}}^n_y.
\end{split}
\end{equation}

Moreover, by~\eqref{olsmdcjp45oyhljXy}, 
$$u(\rho x)\le\sup_{B_{\rho/2}}u\le \frac{C_\star}{\rho^n}\,\int_{{\mathcal{S}}\cap (B_{\rho}\setminus B_{\rho/2})}u(y)\,d{\mathcal{H}}^n_y={C_\star}\,\int_{{\mathcal{S}}_\rho\cap (B_{1}\setminus B_{1/2})}u_\rho(y)\,d{\mathcal{H}}^n_y.$$
This and~\eqref{VATIKE094c1043NMmb098L7b8} yield that
\begin{equation}\label{0mec32VYPSv-4360mcv521.njal} \rho ^{1+s}u(\rho x)
\int_{{\mathcal{S}}\cap (B_1\setminus B_\rho)}\frac{d{\mathcal{H}}^n_y }{|\rho x-y|^{n+1+s}}\le
C_\star\,\widehat{C}\,\int_{{\mathcal{S}}_\rho\cap (B_{1}\setminus B_{1/2})}u_\rho(y)\,d{\mathcal{H}}^n_y.\end{equation}

Gathering this, \eqref{0mec32VYPSv-4360mcv520}, \eqref{0mec32VYPSv-4360mcv521}, and~\eqref{0mec32VYPSv-4360mcv522},
we conclude that\begin{equation}\label{890.o304vm9KSL:DC}|
{\mathcal{L}}_{{\mathcal{S}}_\rho,B_1} u_\rho(x)|\le\left(\frac{\widehat{C}\rho ^{\beta}M}{r_\star^\beta}+C_\star+C_\star\,\widehat{C}
\right)
\int_{{\mathcal{S}}_\rho\cap (B_1\setminus B_{1/2})} u_\rho(y)\,d{\mathcal{H}}^n_y.\end{equation}

Now, we observe that, by~\eqref{09mcBUSLSi024rv83S59876}, in case~1,
\begin{equation*}
\begin{split}&
{\mathcal{L}}_{{\mathcal{S}}_\rho,B_1} v_\rho(x)
=\frac{\rho ^{1+s}}{1-\eta}\Bigg[{\mathcal{L}}_{{\mathcal{S}},B_1} (\overline v-\eta u)(\rho x)\\&\qquad
-\int_{{\mathcal{S}}\cap (B_1\setminus B_\rho)}\frac{ (\overline v-\eta u)(y)}{|\rho x-y|^{n+1+s}}\,d{\mathcal{H}}^n_y
+(\overline v-\eta u)(\rho x)
\int_{{\mathcal{S}}\cap (B_1\setminus B_\rho)}\frac{d{\mathcal{H}}^n_y }{|\rho x-y|^{n+1+s}}\Bigg]
\end{split}\end{equation*}
and a similar identity holds for case~2, leading to
\begin{equation*}
\begin{split}&|
{\mathcal{L}}_{{\mathcal{S}}_\rho,B_1} v_\rho(x)|
\le\frac{\rho ^{1+s}}{1-\eta}\Bigg[|{\mathcal{L}}_{{\mathcal{S}},B_1} \overline v (\rho x)|
+|{\mathcal{L}}_{{\mathcal{S}},B_1} u(\rho x)|
\\&\qquad
+\int_{{\mathcal{S}}\cap (B_1\setminus B_\rho)}\frac{\overline v(y)+ u(y)}{|\rho x-y|^{n+1+s}}\,d{\mathcal{H}}^n_y
+\big(\overline v(\rho x)+u(\rho x)\big)
\int_{{\mathcal{S}}\cap (B_1\setminus B_\rho)}\frac{d{\mathcal{H}}^n_y }{|\rho x-y|^{n+1+s}}\Bigg]\\&
\le\frac{2\rho ^{1+s}}{1-\eta}\Bigg[|{\mathcal{L}}_{{\mathcal{S}},B_1} \overline v (\rho x)|
+|{\mathcal{L}}_{{\mathcal{S}},B_1} u(\rho x)|
\\&\qquad
+\int_{{\mathcal{S}}\cap (B_1\setminus B_\rho)}\frac{ u(y)}{|\rho x-y|^{n+1+s}}\,d{\mathcal{H}}^n_y
+u(\rho x)
\int_{{\mathcal{S}}\cap (B_1\setminus B_\rho)}\frac{d{\mathcal{H}}^n_y }{|\rho x-y|^{n+1+s}}\Bigg]
\end{split}\end{equation*}

We thus use again~\eqref{0mec32VYPSv-4360mcv521}, \eqref{0mec32VYPSv-4360mcv522}, and~\eqref{0mec32VYPSv-4360mcv521.njal} and we find that
\begin{equation*}|
{\mathcal{L}}_{{\mathcal{S}}_\rho,B_1} v_\rho(x)|\le2\left(\frac{\widehat{C}\rho ^{\beta}M}{r_\star^\beta}+C_\star+C_\star\,\widehat{C}
\right)
\int_{{\mathcal{S}}_\rho\cap (B_1\setminus B_{1/2})} u_\rho(y)\,d{\mathcal{H}}^n_y.\end{equation*}
Combing this with~\eqref{890.o304vm9KSL:DC}, we complete the proof of the second claim in~\eqref{smwq908vr9043 85219-BGEAV:L-9143mvn5}.
\end{proof}

\subsection{Regularity theory for coupled equations}

The next important ingredient for the~$C^{1,\alpha}$-regularity of the boundary trace
is to consider how the ratio between the vertical component of the
normal and the tangential components
improves at a smaller scale. To capture this phenomenon, one needs
a regularity theory for the solutions of two similar, but slightly different equations.
These equations need to encode the properties of the different normal components,
but also to maintain a convenient structure to allow iteration at smaller scales.
The coefficients of the equation also need to improve during the iteration,
allowing for enhanced controls of the ratio oscillation in terms of a universal quantity.

\begin{proposition}\label{oqX3dc4grjsd9ouy094h6ypok45otykh}
Let~$\alpha\in(s,1]$ and~${\mathcal{S}}\subseteq\R^{n+1}$ be a bounded, connected hypersurface of class~$C^{1,\alpha}$
such that~$0\in{\mathcal{S}}$. 

Given a set~$V\subseteq\R^{n+1}$, let
$${\mathcal{L}}_{{\mathcal{S}},V }f(x):=\int_{{\mathcal{S}}\cap V}\frac{ f(y)-f(x)}{|x-y|^{n+1+s}}\,d{\mathcal{H}}^n_y.$$

Let~$u$, $ v:{\mathcal{S}}\to\R$, with\begin{equation}\label{odjl903plf-p3tgbXmo2rjfg}
{\mbox{$0\le v\le u$ in~${\mathcal{S}}\cap B_1$.}}\end{equation}

Assume that, in~${\mathcal{S}}\cap B_{1/2}\cap\{u>0\}$,
\begin{equation}\label{OJScvns.31.3dfSOJ02019475}
\begin{split}
&{\mathcal{L}}_{{\mathcal{S}},B_1}u\le M\, u\\
{\mbox{and }}\quad&|{\mathcal{L}}_{{\mathcal{S}},B_1} u|+|{\mathcal{L}}_{{\mathcal{S}},B_1} v|\le M\int_{{\mathcal{S}}\cap (B_1\setminus B_{1/2})}u(y)\,d{\mathcal{H}}^n_y,
\end{split}
\end{equation}
for some~$M>0$.

Assume additionally that there exists~$\delta\in\left(0,\frac1{10}\right)$ such that,
for every~$r\in\left(0,\frac12\right)$,
\begin{equation*}
{\mbox{there exists~$z_r\in{\mathcal{S}}\cap (B_{r/4}\setminus B_{r/8})$
such that~${\mathcal{S}}\cap B_{\delta r}(z_r)\subseteq\{u>0\}$.}}\end{equation*}

Then, there exist~$\eta$, $\rho\in(0,1)$ (depending only on~$n$, $s$, ${{\mathcal{S}}}$, $\delta$, and~$C$,
but independent of~$M$)
and~$\eta_0$, $\rho_0\in(0,1)$, (depending only on~$n$, $s$, ${{\mathcal{S}}}$, $\delta$, $C$,
and~$M$), such that, for all~$k\in\N$, in~${\mathcal{S}}\cap B_{\rho_0\rho^k}$ we have that
\begin{equation}\label{ojsd34otiykh6Yy89j32d57} a_ku\le v\le b_ku,\end{equation}
for some~$a_k$, $b_k\in[0,1]$, satisfying~$b_k\ge a_k$, $a_k$ non-decreasing, $b_k$ non-increasing, and
$$ b_{k}-a_k= (1-\eta_0)\,(1-\eta)^{k}.$$
\end{proposition}

\begin{proof} We use Lemma~\ref{LEM5.1} (with~$\overline v:=v$ here)
and we see that, \begin{equation}\label{PRImAas}
{\mbox{in~${\mathcal{S}}\cap B_{\rho_0}$,
either~$\eta_0 u\le v\le u$ or~$0\le v\le(1-\eta_0)u$.}}\end{equation}

Now we argue recursively. The basis of the induction is supplied by~\eqref{PRImAas},
which gives the desired result with either~$a_{0}:=\eta_0$ and~$b_{0}:=1$,
or~$a_{0}:=0$ and~$b_{0}:=1-\eta_0$. In any case, $b_{0}-a_{0}=1-\eta_0$,
and this provides the basis of the induction argument.\footnote{The parameters~$\eta_0$ and~$\rho_0$ are implicitly assumed to be bounded by structural constants, whenever convenient. \label{SMARH}
As a notational remark,
we are using here the ``$0$'' subscript in the parameters~$\eta_0$
and~$\rho_0$ to emphasize their dependence on~$M$, interestingly
this dependence only takes place in the first step of the iteration. Indeed,
the parameters~$\eta$
and~$\rho$ in the next steps of the iteration will come from Lemma~\ref{LEM5.1},
used with~$M:=\widehat C$, for a structural constant~$\widehat C$.
The fact that~$\eta$ can be chosen independently of~$M$ is essential to our construction. Indeed, rescaling the equation on a small ball produces a large factor (the radius raised to the power~$-1-s$), which would otherwise become uncontrollable in the iteration scheme.}

Now we suppose that the desired result is true for the index~$k\in\N$ and we aim
at proving it for the index~$k+1$. To this end, we\footnote{This proof would have
worked equally well if we defined instead~$\widetilde v:=\frac{b_k u-v}{b_k-a_k}$.} define~$\widetilde v:=\frac{v-a_k u}{b_k-a_k}$.
We point out that~$0\le\widetilde v\le u$ in~${\mathcal{S}}\cap B_{\rho_0\rho^{k}}$, due to the inductive hypothesis in~\eqref{ojsd34otiykh6Yy89j32d57}.

Let also~$\overline{v}(x):=\widetilde{v}(\rho_0\rho^{k}x)$
and~$\overline{u}(x):=u(\rho_0\rho^{k}x)$.
By~\eqref{EAFYJLDIKSPpwk3mbm01LK-vvNMASP.2},
we know that~$0\le\overline{v}\le\overline u$
in~${\mathcal{S}}_k\cap B_{1}$, being~${\mathcal{S}}_k$ the
zoom-in of~${\mathcal{S}}$ by a factor of~$
\rho_0\rho^{k}$ (in particular, the $C^{1,\alpha}$-regularity of~${\mathcal{S}}_k$ is controlled by that of~${\mathcal{S}}$).

Additionally,  by~\eqref{smwq908vr9043 85219-BGEAV:L-9143mvn5},
in~${\mathcal{S}}_k\cap B_{1/2}\cap\{\overline u>0\}$,
\begin{equation*}
\begin{split}&
{\mathcal{L}}_{{\mathcal{S}}_k,B_1} \overline u\le \widehat{C}\,\overline u\\
{\mbox{and }}\quad&
|{\mathcal{L}}_{{\mathcal{S}}_k,B_1} \overline u|+|{\mathcal{L}}_{{\mathcal{S}}_k,B_1} \overline v|\le \widehat{C}\int_{{\mathcal{S}}_k\cap (B_1\setminus B_{1/2})}\overline u(y)\,d{\mathcal{H}}^n_y.\end{split}\end{equation*}

On this account, we can use Lemma~\ref{LEM5.1} (now with~$M:=\widehat C$, as anticipated
in footnote~\ref{SMARH}), concluding that, in~${\mathcal{S}}\cap B_{\rho}$,
either~$\eta \overline u\le \overline v$ or~$\overline v\le(1-\eta)\overline u$.
This leads to the completion of the inductive step
(in the first case, by defining~$a_{k+1}:=a_k+\eta(b_k-a_k)$ and~$b_{k+1}:=b_k$,
in  the second case, by defining~$a_{k+1}:=a_k$ and~$b_{k+1}:=b_k-
\eta(b_k-a_k)$).
\end{proof}

\subsection{$C^{1,\gamma}$-regularity of the
trace of nonlocal minimal surfaces at points of stickiness}

The goal of this section is to use the material developed so far to complete
the proof of Theorem~\ref{THM1}.

\begin{proof}[Proof of Theorem~\ref{THM1}]
Let~$x_0$ be as in the statement
of Theorem~\ref{THM1}. Up to a dilation,
we can suppose that 
$$ \lim_{\Omega\ni x'\to x_0} u_E(x')\ge4+\lim_{\R^n\setminus\Omega\ni x'\to x_0} u_0(x').$$
Hence, possibly up to an additional dilation,
we can suppose that for all~$x_\star\in \Sigma\cap B_1(x_0)$
we have that~$\Sigma\cap B_2(x_\star)$ is
a connected hypersurface of class~$C^{1,\frac{1+s}2}$, see Figure~\ref{f8103pra3g5e7.1}.

\begin{figure}[htbp]
        \centering
         \includegraphics[width=0.35\linewidth]{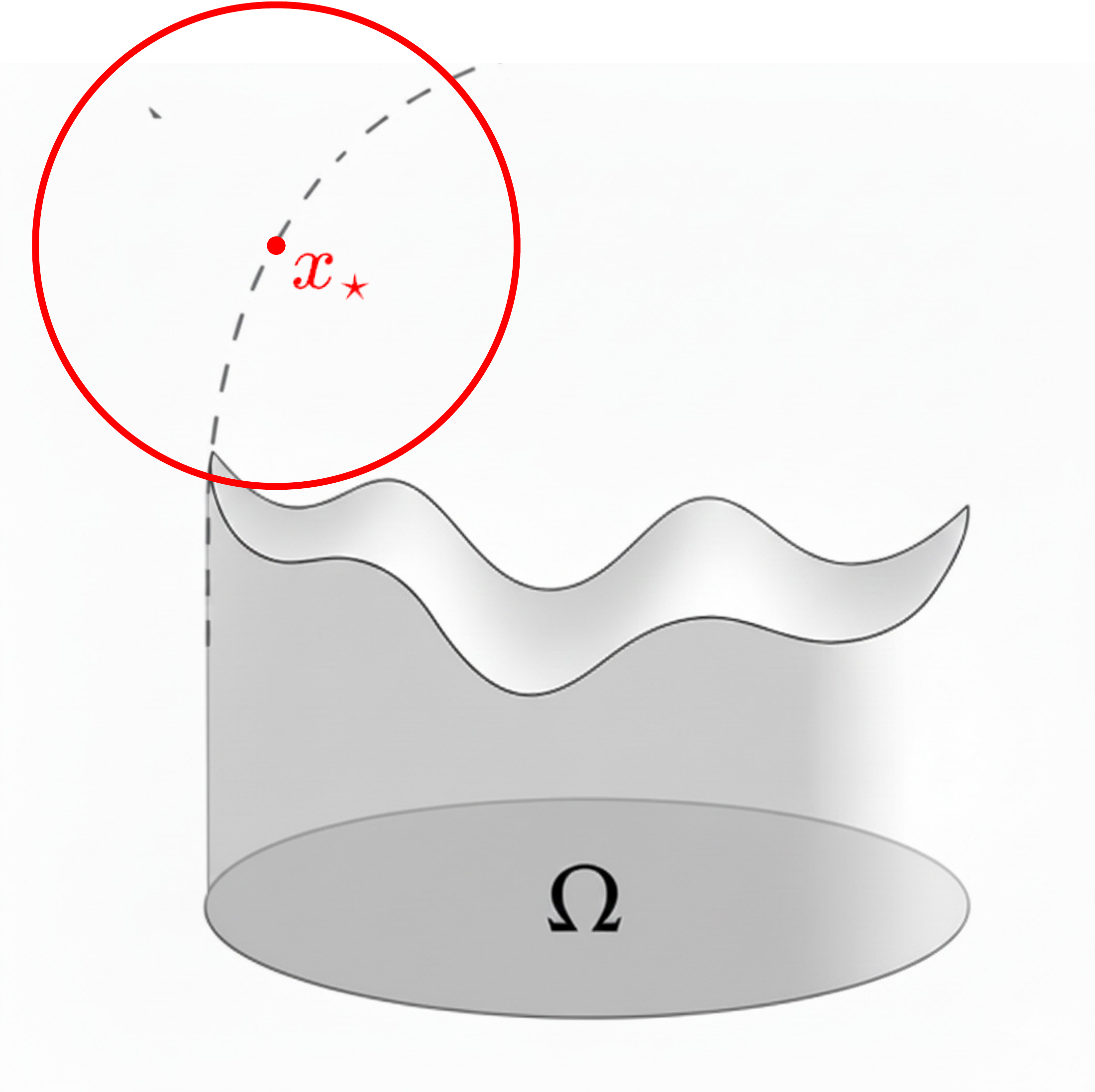}$\qquad\qquad$
        \includegraphics[width=0.35\linewidth]{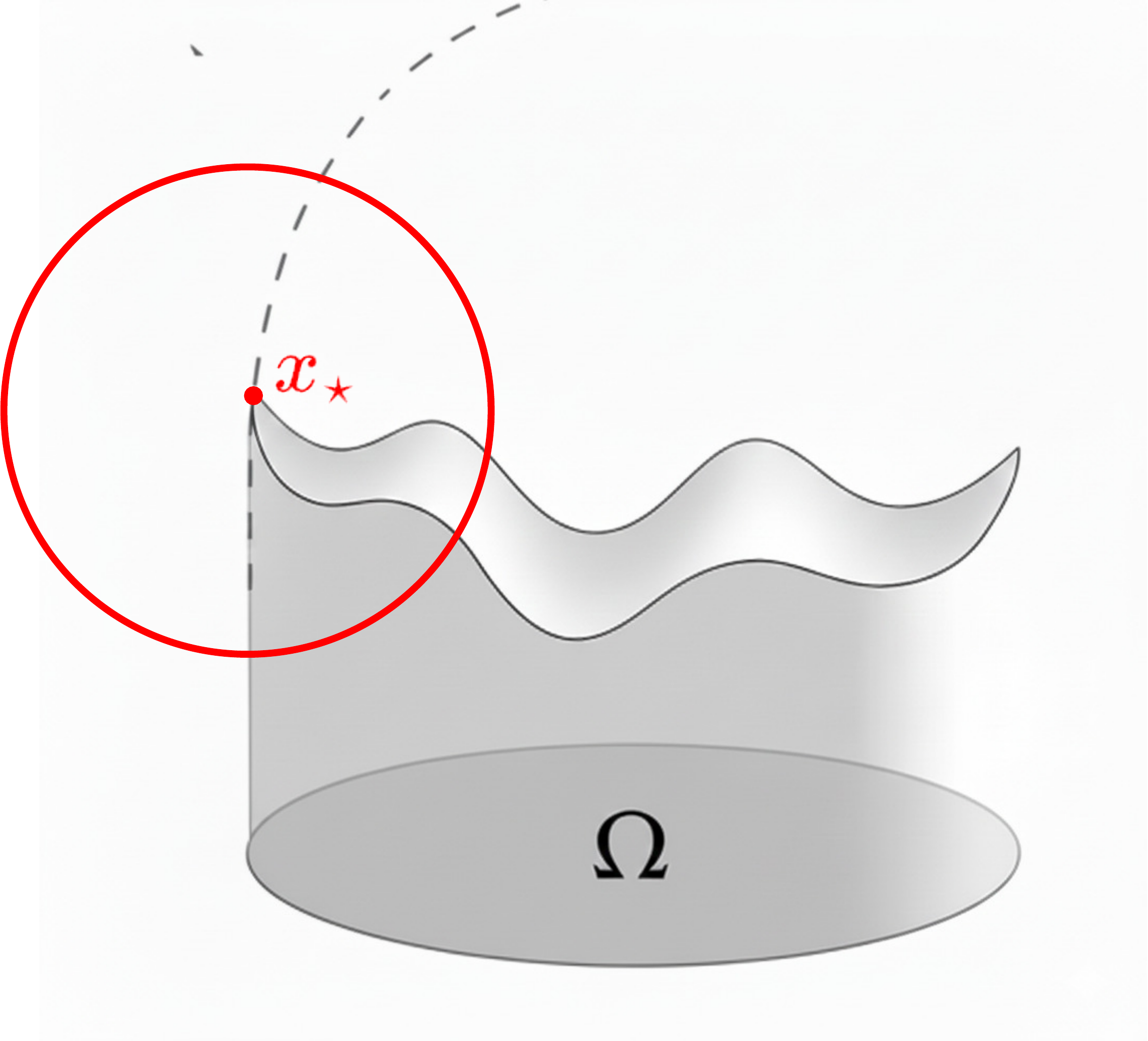}
        \caption{\footnotesize\sl Sketch of the position of the point~$x_\star$ to obtain
        uniform oscillation bounds for the normal ratio up to the boundary.}
        \label{f8103pra3g5e7.1}
    \end{figure}

This allows us to use Proposition~\ref{oqX3dc4grjsd9ouy094h6ypok45otykh} with~$u:=2C\nu_{n+1}$ and~$v:=C\nu_{n+1}+\tau\cdot\nu$,
with~$C$ as in Corollary~\ref{ojsndofkw34rtgoCIr7} (this entails that~$0\le v\le u$ in this setting).
In this way, we
conclude that, for large~$k$ and suitable~$C_\star>0$ and~$\epsilon\in(0,1)$,
$$ \mathop{\mathrm{osc}}\limits_{B_{2^{-k}(x_\star)}}\frac{C\nu_{n+1}+\tau\cdot\nu}{2C\nu_{n+1}}\le C_\star\,(1-\epsilon)^k$$
and therefore
$$ \mathop{\mathrm{osc}}\limits_{B_{2^{-k}(x_\star)}}\frac{\tau\cdot\nu}{\nu_{n+1}}\le 2CC_\star\,(1-\epsilon)^k.$$

This estimate %(see e.g.~\cite[Lemma~8.23]{MR1814364})
guarantees the H\"older decay of the oscillation of~$\frac{\tau\cdot\nu}{\nu_{n+1}}$ at the point~$x_\star$
(with H\"older exponent~$\log_2\frac1{1-\epsilon}$). Notice that the estimate obtained
is uniform in~$x_\star$ up to the boundary, yielding that~$\frac{\tau\cdot\nu}{\nu_{n+1}}$ is H\"older
continuous in a neighborhood of any point of stickiness~$x_0$.

The desired result now follows from Lemma~\ref{L2sw2I2C32rA}.
\end{proof}

\section{Differentiability of nonlocal minimal graphs at points of stickiness and proof of Theorem~\ref{2swa1l2e4oj2dn2fl3er4t5e4n2v2}}\label{M324r3tgRguGJKHDOJCDL417872152}

As discussed in~\cite[page~127]{MR4178752},
the main obstacle in proving the differentiability of nonlocal minimal graphs at points of stickiness is to avoid vertical tangents of the trace, since one needs to consider
points of the trace at which the gradient is maximal.
When~$n=2$, the matter was settled in~\cite{MR4178752}
by looking at conical blow-up limits,
because the conical trace would
only identify a single slope in the plane, and one needs in this case
only to exclude ``completely vertical'' blow up limits.
When~$n\ge3$, the situation is, in principle, significantly
more complex, since a vertical tangency of the trace at a single point
would produce an infinite slope without the limit cone becoming completely
vertical. In this spirit, the role of Theorem~\ref{THM1} is precisely
to exclude vertical tangents in the trace (plus,
to ensure sufficient compactness).

As a matter of fact (see the proof of Theorem~1.4 in~\cite{MR4178752}),
Theorem~\ref{2swa1l2e4oj2dn2fl3er4t5e4n2v2} would follow once we prove the following
result for cones:

\begin{theorem}[Triviality
of nonlocal minimal cones in halfspaces]\label{posjwxhe0itvnu945opmbiytv9cmir23Hnq72Gnm2yF}
Let~$u_\star:\R^n\to\R$ be an $s$-minimal graph in~$\R^{n-1}\times(0,+\infty)$. 

Assume that~$u_\star(x'',x_n)=0$ for all~$(x'',x_n)\in\R^{n-1}\times(-\infty,0)$ and that~$u_\star$
is positively homogeneous of degree~$1$ (i.e., $u_\star(tx')=tu_\star(x')$ for every~$x'\in\R^n$
and~$t>0$).

%Suppose that
%\begin{equation}\label{QDWFFM3R5481W9UJDC252}
%\lim_{x'\to0}u_\star(x')=0.
%\end{equation}

Then $u_\star(x')=0$ for all~$x'\in\R^n$.
\end{theorem}

This result was proved in~\cite[Theorem~1.5]{MR4178752}
in the case~$n=2$. 
The establishment of this result in any dimension addresses a question posed
in~\cite[Open Problem~1.7]{MR4178752}.
In a nutshell, Theorem~\ref{posjwxhe0itvnu945opmbiytv9cmir23Hnq72Gnm2yF}
states that a nonlocal minimal cone with graphical structure
in a halfspace and trivial datum in the complementary halfspace is necessarily trivial.

To prove Theorem~\ref{posjwxhe0itvnu945opmbiytv9cmir23Hnq72Gnm2yF}
(and thus Theorem~\ref{THM1})
we will rely on a structural continuity result, as described here below.

\subsection{Structural continuity at stickiness points}

We already know from~\cite{MR3516886} that nonlocal minimal graphs
are continuous up to the boundary from the interior of the domain.
We want now to obtain a bound on the modulus of continuity
at stickiness points that is ``universal'',
namely that only depends on the size of the boundary discontinuity,
$n$, and~$s$. This is important because
this information will pass to the limit through a sequence of nonlocal minimal graphs
and will ensure in Section~\ref{9sxjhwmr8c t940536uxj} a uniform detachment from the boundary in a ``renormalized'' picture
and prevent the formation of vertical boundary segments
due to limiting processes (see Figure~\ref{f8103pra3g5e7.1ascvfgasncsd}).
%The uniformity of this bound will in turn allow us
%to use the boundary Harnack inequality with structural constants, not depending
%on the specific surfaces involved in the limit process.

In this setting, the structural continuity result that we need goes as follows:

\begin{lemma}\label{by601widkpw304tyhjbl}
Let~$u_\star:\R^n\to\R$ be an $s$-minimal graph in a set~$\Omega_\star$
of class~$C^{1,1}$,
with~$0\in\partial\Omega_\star$.

Suppose that there exists~$\delta>0$ such that~$u_\star(x')=0$ for all~$x'\in\R^n\setminus\Omega_\star$
with~$|x'|<\delta$.

Suppose also that there exist~$a>0$ and a sequence of points~$p'_k\in\Omega_\star$
such that, as~$k\to+\infty$,
$$ p'_k\to 0\qquad{\mbox{and}}\qquad u_\star(p'_k)\to a.$$

Then, there exist~$\delta_0\in(0,1)$ and a modulus of continuity~$\omega_0$
(i.e., an increasing function~$\omega_0:[0,1]\to[0,1]$, continuous at~$0$ and such that~$\omega_0(0)=0$)
such that, for all~$x',y'\in\Omega_\star\cap B_{\delta_0}$, we have that
$$|u_\star(x')-u_\star(y')|\le\omega_0(|x'-y'|).$$
Here, $\delta_0$ and~$\omega_0$ depend only on~$n$, $s$, $\Omega_\star$,
$\delta$, and~$a$(and they are independent of~$u_\star$).
\end{lemma}

\begin{proof} Suppose not. Then, we find~$s$-minimal graphs~$u_j$
in~$\Omega_\star$ such that~$u_j(x')=0$ for all~$x'\in\R^n\setminus\Omega_\star$
with~$|x'|<\delta$ and for which
one can find sequences of points~$p'_{j,k}\in\Omega_\star$
such that
\begin{equation}\label{GG-qpHdkprghk-ypt} \lim_{k\to+\infty}p'_{j,k}= 0\qquad{\mbox{and}}\qquad\lim_{k\to+\infty} u_j(p'_{j,k})= a,\end{equation}
and also points~$\zeta'_{j},\eta'_{j}\in\Omega_\star\cap B_{1/j}$ such that
$$ u_j(\zeta'_{j})-u_j(\eta'_{j})\ge b,$$for some~$b>0$.

By~\eqref{GG-qpHdkprghk-ypt} and~\cite{MR3532394},
we have that the surfaces~$\Sigma_j:=\{x_{n+1}=u_j(x')\}$
are of class~$C^{1,\frac{1+s}2}$,
uniformly in~$j$, in~$B_{\delta_0}(0,\dots,0,a)\cap\Omega_\star$, for some~$\delta_0\in(0,a)$,
with a graphical structure in the interior normal direction of~$\Omega_\star\times\R$.

Hence, up to a subsequence, we can suppose that~$\Sigma_j$ in~$
B_{\delta_0}(0,\dots,0,a)\cap\Omega_\star$ approaches a graphical hypersurface~$\Sigma_\infty$
in the~$C^{1,\alpha}$-sense, with~$\alpha>s$
(also, we can suppose that~$\Sigma_j$ approaches
a nonlocal minimal graph~$\Sigma_\infty$ everywhere in the measure theoretic sense
and with respect to the Hausdorff distance, see~\cite[Appendices~A and~C]{MR4104542}).

Thus, if~$P_j:=(\zeta'_j,u_j(\zeta'_{j}))$ and~$Q_j:=(\eta'_j,u_j(\eta'_{j}))$, we have that~$P_j\to P$
and~$Q_j\to Q$, with~$P$ and~$Q$ vanishing points for the nonlocal mean curvature of~$\Sigma_\infty$,
that is
$$
\int_{\R^{n+1}}\frac{\widetilde\chi_{E_\infty}(y)}{|P-y|^{n+1+s}}\,dy=0
=\int_{\R^{n+1}}\frac{\widetilde\chi_{E_\infty}(y)}{|Q-y|^{n+1+s}}\,dy,$$
where~$ E_\infty$ is the subgraph of~$\Sigma_\infty$
and
$$ \widetilde\chi_{E_\infty}(y)=\begin{cases}
1&{\mbox{ if }}y\in\R^{n+1}\setminus{E_\infty},\\
-1&{\mbox{ if }}y\in{E_\infty}.
\end{cases}$$

We observe that, by construction,~$P=(0,\sigma)\in\R^n\times\R$ and~$Q=(0,\tau)\in\R^n\times\R$
with~$\sigma-\tau\ge b$.
Therefore, setting~$\widetilde E_\infty:=E_\infty+(0,\dots,0,\sigma-\tau)$, the graphical structure entails that~$\widetilde E_\infty\supseteq E_\infty$ and
$$ \widetilde\chi_{E_\infty}(y)-\widetilde\chi_{\widetilde E_\infty}(y)=\begin{cases}
2&{\mbox{ if }}y\in\widetilde E_\infty\setminus{E_\infty},\\
0&{\mbox{ otherwise. }}
\end{cases}$$

All in all, since~$P=Q+(0,\dots,0,\sigma-\tau)$,
if~$C_\delta:=\{(x',x_n)$ s.t. $x'\in\R^n\setminus\Omega_\star$
and~$|x'|<\delta\}$,
we have that
\begin{eqnarray*}
0&=&\int_{\R^{n+1}}\frac{\widetilde\chi_{E_\infty}(y)}{|P-y|^{n+1+s}}\,dy
-\int_{\R^{n+1}}\frac{\widetilde\chi_{E_\infty}(y)}{|Q-y|^{n+1+s}}\,dy\\
&=&\int_{\R^{n+1}}\frac{\widetilde\chi_{E_\infty}(y)}{|P-y|^{n+1+s}}\,dy
-\int_{\R^{n+1}}\frac{\widetilde\chi_{\widetilde E_\infty}(y)}{|P-y|^{n+1+s}}\,dy\\&
=&\int_{\widetilde E_\infty\setminus{E_\infty}}\frac{2\,dy}{|P-y|^{n+1+s}}\\&\ge&
\int_{(\widetilde E_\infty\setminus{E_\infty})\cap C_\delta}\frac{2\,dy}{|P-y|^{n+1+s}}\\&=&
\int_{(\{ |x'|<\delta\}\setminus\Omega_\star)\times\{x_n\in(0,\sigma-\tau)\}}\frac{2\,dy}{|P-y|^{n+1+s}}\\&>&0.
\end{eqnarray*}We have thereby reached the desired contradiction.
\end{proof}

\subsection{Triviality
of nonlocal minimal cones}\label{9sxjhwmr8c t940536uxj}

We can now address the proof of
Theorem~\ref{posjwxhe0itvnu945opmbiytv9cmir23Hnq72Gnm2yF}.

\begin{proof}[Proof of Theorem~\ref{posjwxhe0itvnu945opmbiytv9cmir23Hnq72Gnm2yF}]
The desired result is proved if we show that~$\nu_i$ vanishes identically
in~$\{x_n>0\}$, for all~$i\in\{1,\dots,n\}$. In fact, by homogeneity,
it suffices to show that~$\nu_i(x)=0$ for all~$x\in \Sigma\cap S^n\cap\{x_n>0\}$
and all~$i\in\{1,\dots,n\}$.

We will focus on the first coordinate, the others being completely analogous, namely
we aim at showing that
\begin{equation}\label{9jqsdcw0gpkbmpijr893vny04pw0e-f1oled.01}
{\mbox{$\nu_1(x)=0$ for all~$x\in \Sigma\cap S^n\cap\{x_n>0\}$.}}\end{equation} We define
$$ \sigma_1:=\sup_{ \Sigma\cap S^n\cap\{x_n> 0\}}\frac{\nu_1}{\nu_{n+1}}.$$
At this level, we do not know that~$\sigma_1$ is finite
(for example, we cannot use Corollary~\ref{ojsndofkw34rtgoCIr7} at this stage,
since, to apply it, we would need
a full neighborhood where the nonlocal minimal surface sticks to the cylinder).

Nonetheless, we claim that
\begin{equation}\label{9jqsdcw0gpkbmpijr893vny04pw0e-f1oled.02}
\sigma_1\le0.
\end{equation}
Notice that, once this is proven, we conclude that~$\nu_1\le0$ in~$\Sigma\cap S^n\cap\{x_n> 0\}$.
Then, up to exchanging~$\nu_1$ with~$-\nu_1$, it also follows
that~$-\nu_1\le0$ in~$\Sigma\cap S^n\cap\{x_n> 0\}$, and therefore~\eqref{9jqsdcw0gpkbmpijr893vny04pw0e-f1oled.01} is established.

Hence, we focus on the proof of~\eqref{9jqsdcw0gpkbmpijr893vny04pw0e-f1oled.02}.
For this, we consider a sequence~$X_k\in  \Sigma\cap S^n\cap\{x_n> 0\}$ such that~$\frac{\nu_1(X_k)}{\nu_{n+1}(X_k)}\to \sigma_1$ as~$k\to+\infty$.

Up to a subsequence, we have that~$X_k\to X_\infty\in S^n$.
We also let~$R_k$ be the distance of~$X_k$ to~$\{x_{n+1}=0\}\cap\{x_n\le0\}$ and we let~$Y_k\in
\{x_{n+1}=0\}\cap\{x_n\le0\}$ be a point realizing this distance.

We write~$Y_k=(Y_k',Y_{k,n},0)\in\R^{n-1}\times\R\times\R$ and observe that, without loss of generality, we can suppose that~$Y_{k,n}=0$, because
\begin{eqnarray*}&&|(Y_k',0,0)-X_k|^2=|Y_k-X_k|^2+X_{k,n}^2
-(Y_{k,n}-X_{k,n})^2\\&&\quad=|Y_k-X_k|^2+2X_{k,n}Y_{k,n}
-Y_{k,n}^2\le|Y_k-X_k|^2.\end{eqnarray*}

Similarly, we can suppose that~$Y_k'=X_k'$, because
\begin{eqnarray*}&& |(X_k',0,0)-X_k|^2=X_{k,n}^2+X_{k,n+1}^2\\&&\quad
\le|X_k'-Y_k'|^2+
X_{k,n}^2+X_{k,n+1}^2
=
|Y_k-X_k|^2.\end{eqnarray*}

Notice also that~$|Y_k|\le|Y_k-X_k|+|X_k|\le|0-X_k|+|X_k|=2$,
so, up to a subsequence, we suppose that~$Y_k\to Y_\infty$
and~$R_k\to R_\infty\in[0,1]$.

We define~$E_k:=\frac{E-Y_k}{R_k}$
and~$\Sigma_k:=\partial E_k$.
Up to a subsequence, we know that~$E_k$
converges to a nonlocal minimal set~$E_\infty$
in~$L^1_{\rm loc}(\R^{n+1})$, see~\cite[Appendix~A]{MR4104542}
(and, while~$\Sigma_k$ is a cone for every~$k$, we do not know that~$\Sigma_\infty$
is a cone, since the base point~$Y_k$ of the dilation defining~$\Sigma_k$ varies with~$k$).
Moreover, the above local convergence holds true also
in the Hausdorff distance, see~\cite[Appendix~C]{MR4104542},
and locally in~$C^{1,\frac{1+s}2}$ in the vicinity of
points of stickiness for~$\Sigma_\infty:=\partial E_\infty$,
due to the improvement of flatness result in~\cite{MR3532394}.

We let~$p_k:=\frac{X_k-Y_k}{R_k}$ and we stress that~$|p_k|=\frac{|X_k-Y_k|}{R_k}=1$,
hence, up to a subsequence, we have that~$p_k\to p_\infty\in S^n$.

We claim that
\begin{equation}\label{X2.2ISJOLPXpaP}
p_\infty\in\{x_n>0\}.
\end{equation}
Once this is established, the proof of~\eqref{9jqsdcw0gpkbmpijr893vny04pw0e-f1oled.01}
is completed through the following argument. It follows from~\eqref{X2.2ISJOLPXpaP}
and the interior regularity for nonlocal minimal graphs (see~\cite{MR3934589})
that~$\Sigma_k$ approaches~$\Sigma$ in~$C^{1,\alpha}$ in the vicinity of~$p_\infty$,
with~$\alpha>0$. Accordingly,
denoting by~$\nu_k$ the normal to~$\Sigma_k$
and by~$\nu_\infty$ the normal to~$\Sigma_\infty$, we have that
$$ \lim_{k\to+\infty}
\frac{\nu_{k,1}(p_k)}{\nu_{k,n+1}(p_k)}=
\frac{\nu_{\infty,1}(p_\infty)}{\nu_{\infty,n+1}(p_\infty)}.$$
As a result, for all~$\ell\in\N$ and any~$z\in \Sigma_\ell\cap\{x_n> 0\}$,
\begin{equation}\label{doqjd0rti540iitkh43f0vl32fc2dcXwfvXce3}\begin{split}&\frac{\nu_{\ell,1}(z)}{\nu_{\ell,n+1}(z)}
\le\sup_{ \Sigma_\ell\cap\{x_n> 0\}}\frac{\nu_{\ell,1}}{\nu_{\ell,n+1}}
=\sup_{ \Sigma\cap\{x_n> 0\}}\frac{\nu_{1}}{\nu_{n+1}}
=\sup_{ \Sigma\cap S^n\cap\{x_n> 0\}}\frac{\nu_1}{\nu_{n+1}}\\&\quad=\sigma_1=
\lim_{k\to+\infty}\frac{\nu_1(X_k)}{\nu_{n+1}(X_k)}=
\lim_{k\to+\infty}
\frac{\nu_{k,1}(p_k)}{\nu_{k,n+1}(p_k)}=\frac{\nu_{\infty,1}(p_\infty)}{\nu_{\infty,n+1}(p_\infty)}.
\end{split}\end{equation}
Notice that the above chain of inequalities remains true if~$\sigma_1$ is infinite:
but this situation can now be ruled out, since otherwise~$\nu_{\infty,n+1}(p_\infty)=0$, with~$p_\infty$
as in~\eqref{X2.2ISJOLPXpaP}, which contradicts the graphical structure of~$E_\infty$.

Passing now~\eqref{doqjd0rti540iitkh43f0vl32fc2dcXwfvXce3} to the limit as~$\ell\to+\infty$, we conclude that, for all~$z\in \Sigma_\infty\cap\{x_n> 0\}$,
\begin{equation}\label{doqjd0rti540iitkh43f0vl32fc2dcXwfvX}
\frac{\nu_{\infty,1}(z)}{\nu_{\infty,n+1}(z)}\le\frac{\nu_{\infty,1}(p_\infty)}{\nu_{\infty,n+1}(p_\infty)}.
\end{equation}

We denote by~$u_\infty$ the function representing the graph of~$\Sigma_\infty$
and we observe that, by construction,~$u_\infty$ vanishes identically in~$\R^{n-1}\times(-\infty,0)$.
Since~$\frac{\nu_{\infty,1}}{\nu_{\infty,n+1}}$ coincides with~$-\partial_1u_\infty$, we have thus discovered in~\eqref{doqjd0rti540iitkh43f0vl32fc2dcXwfvX}
that~$\partial_1u_\infty$ attains its minimum
in~$\R^{n-1}\times(0,+\infty)$ at the point~$p_\infty$. This and~\cite[Lemma~3.8]{MR4178752}
(used here with~$\Omega:=\R^{n-1}\times(0,+\infty)$
and~${\mathcal{B}}$ a ball contained in~$\R^{n-1}\times(-\infty,-1)$) yield that~$\partial_1 u_\infty(p_\infty)\ge0$.

This analysis gives that~$\sigma_1=\frac{\nu_{\infty,1}(p_\infty)}{\nu_{\infty,n+1}(p_\infty)}=
-\partial_1 u_\infty(p_\infty)\le0$, which establishes~\eqref{9jqsdcw0gpkbmpijr893vny04pw0e-f1oled.02}
if~\eqref{X2.2ISJOLPXpaP} holds true.

Hence, to complete the proof of~\eqref{9jqsdcw0gpkbmpijr893vny04pw0e-f1oled.02},
we now dive into the proof of~\eqref{X2.2ISJOLPXpaP}. In this scenario,
we can assume that
\begin{equation}\label{xpjwcq0gibtjm9uhgvc87trds3decd}
\sigma_1\in(0,+\infty],
\end{equation}
otherwise~\eqref{9jqsdcw0gpkbmpijr893vny04pw0e-f1oled.02} holds true and we are done anyway
(and notice that, in this situation, since we do not have~\eqref{9jqsdcw0gpkbmpijr893vny04pw0e-f1oled.02},
we do not know that~$\sigma_1$ is finite to start with).

To check~\eqref{X2.2ISJOLPXpaP}, we argue for the sake of contradiction,
supposing that~$p_\infty\in\{x_n=0\}$.
Since
$$ p_k=\frac{X_k-(X_k',0,0)}{\sqrt{X_{k,n}^2+X_{k,n+1}^2}}
=\frac{(0,X_{k,n},X_{k,n+1})}{\sqrt{X_{k,n}^2+X_{k,n+1}^2}}
,$$
in this scenario we have that
$$ 0=p_{\infty,n}=\lim_{k\to+\infty}p_{k,n}=
\lim_{k\to+\infty}\frac{X_{k,n}}{\sqrt{X_{k,n}^2+X_{k,n+1}^2}}
$$
and therefore~$X_{k,n}\to0$, giving that
$$|p_{\infty,n+1}|=\lim_{k\to+\infty}|p_{k,n+1}|=
\lim_{k\to+\infty}\frac{|X_{k,n+1}|}{\sqrt{X_{k,n}^2+X_{k,n+1}^2}}=1.$$

We thus suppose that~$p_{\infty,n+1}=1$, the case~$p_{\infty,n+1}=-1$
being analogous. Therefore~$p_\infty=(0,\dots,0,1)=:e_{n+1}$.

\begin{figure}[htbp]
        \centering
         \includegraphics[width=0.45\linewidth]{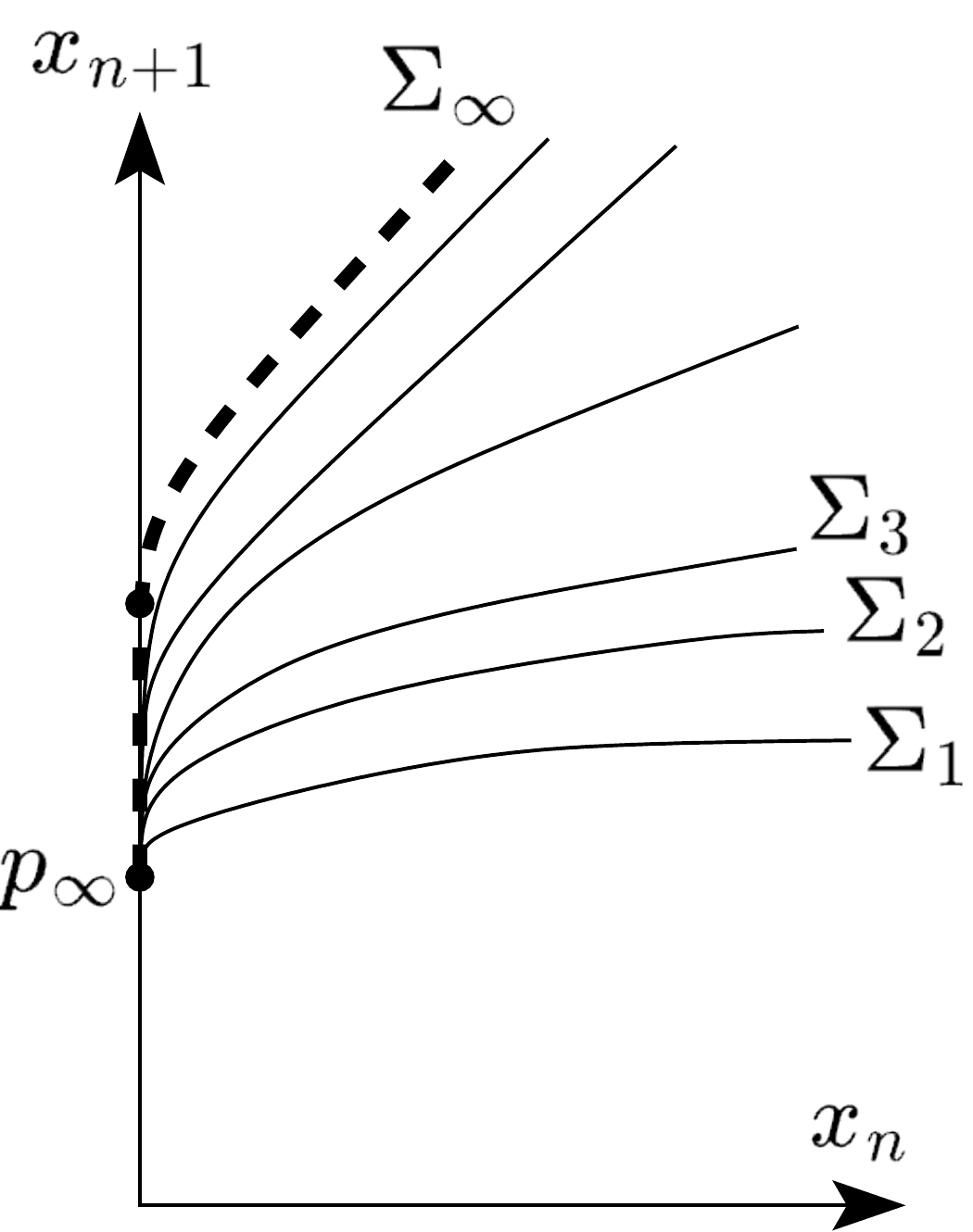}
        \caption{\footnotesize\sl Sketch illustrating what~\eqref{by601widkpw304tyhjb} wants to avoid.}
        \label{f8103pra3g5e7.1ascvfgasncsd}
    \end{figure}

We claim that\footnote{The claim in~\eqref{by601widkpw304tyhjb}
is somewhat subtle. In principle, what could
happen here is that the surface~$\Sigma_k$ might become ``more and more vertical''
near~$p_\infty$ and develop a vertical segment over~$p_\infty$ in the limit,
see Figure~\ref{f8103pra3g5e7.1ascvfgasncsd}.}
\begin{equation}\label{by601widkpw304tyhjb}
{\mbox{$p_\infty$ is an accumulation point for~$\Sigma_\infty$ from~$\{x_n>0\}$,}}\end{equation}
namely that there exists a sequence~$q_k'\in\R^{n-1}\times(0,+\infty)$ such that~$q_k'\to0$
and~$u_\infty(q_k')\to1$.

Indeed, by Lemma~\ref{by601widkpw304tyhjbl}, we know that the function~$u_k$
describing the nonlocal minimal graph~$\Sigma_k$ converges uniformly to~$u_\infty$
in~$\{|x'|\le\delta_0\}\cap\{x_n>0\}$. Hence, since~$p_{k,n+1}=u_k(p_k')\to1$
and~$p'_k\to0$, we find that
$$ 1=\lim_{k\to+\infty} u_k(p_k')=\lim_{k\to+\infty} u_\infty(p_k').$$
This proves~\eqref{by601widkpw304tyhjb}.

As a consequence of~\eqref{by601widkpw304tyhjb},
we have that
\begin{equation}\label{GH-qc0-m5Xjskmpwq:kd:877}
{\mbox{$p_\infty$ is a boundary discontinuity point for~$\Sigma_\infty$.}}\end{equation}
Thus, by~\cite{MR3532394},
we have that~$\Sigma_k$ converges to~$\Sigma_\infty$ in~$C^{1,\frac{1+s}2}$ in a neighborhood
of~$p_\infty$. As a result, as in~\eqref{doqjd0rti540iitkh43f0vl32fc2dcXwfvXce3}
and~\eqref{doqjd0rti540iitkh43f0vl32fc2dcXwfvX}, we have that,
for all~$z\in \Sigma_\infty\cap\{x_n> 0\}$,
\begin{equation}\label{cd0iksqp032tugjb230-ifwjv0vvtnh}
\frac{\nu_{\infty,1}(z)}{\nu_{\infty,n+1}(z)}\le
\lim_{ \Sigma_\infty\cap\{x_n> 0\}\ni p\to p_\infty}
\frac{\nu_{\infty,1}(p)}{\nu_{\infty,n+1}(p)}=\sigma_1.
\end{equation}

The result in~\eqref{GH-qc0-m5Xjskmpwq:kd:877} is also important because
it allows us, at this point, to use Corollary~\ref{ojsndofkw34rtgoCIr7}, and conclude that~$\sigma_1$ is finite.

We thus set~$U:=\sigma_1\,\nu_{\infty,n+1}-\nu_{\infty,1}$ and we know that~$U\ge0$
in~$\Sigma_\infty\cap\{x_n> 0\}$.
Also, due to the regularity of~$\Sigma_\infty$ in a neighborhood
of~$p_\infty$, we can take~$\rho_0\in(0,1)$ such that~$\Sigma_\infty\cap B_{2\rho_0}(p_\infty)$ is of
class~$C^{1,\frac{1+s}2}$.

We let
$${\mathcal{L}} f(x):=\int_{\Sigma_\infty\cap B_{\rho_0}(p_\infty)}\frac{ f(y)-f(x)}{|x-y|^{n+1+s}}\,d{\mathcal{H}}^n_y$$
and we claim that, in ${\Sigma_\infty\cap B_{\rho_0/2}(p_\infty)}\cap\{x_n>0\}$,
\begin{equation}\label{vdeOJSN.DJLM0w3oekfvf-023e-1}
{\mathcal{L}}U< \ell U,\end{equation}
for a bounded function~$\ell$. Indeed, in the notation of
Lemma~\ref{CUTLE}, we deduce from~\eqref{poqljf34o53i6ujp6ui-2pro345lty}
that, for all~$x\in\Sigma_\infty\cap B_{\rho_0/2}(p_\infty)\cap\{x_n>0\}$,
\begin{equation}\label{vdeOJSN.DJLM0w3oekfvf-023e-1k}
{\mathcal{L}} U(x)\le{\mathcal{L}} U(x)
+U(x)\int_{{\Sigma_\infty\cap B_{\rho_0}(p_\infty)}}\frac{1-\nu_\infty(x)\cdot\nu_\infty(y)}{|x-y|^{n+1+s}}\,d{\mathcal{H}}^n_y=a(x)U(x)+f(x),
\end{equation}
for a bounded function~$a$ and, in view of~\eqref{oqjdcmwerlbX2},
if~$\Sigma_\infty=\partial E_\infty$,
\begin{equation*}\begin{split}& f(x)=-\int_{{\partial^*E_\infty\setminus B_{\rho_0}(p_\infty)}}\frac{U(y)}{|x-y|^{n+1+s}}\,d{\mathcal{H}}^n_y\le
-\int_{{\partial^*E_\infty\cap\{x_n<-2\}}}\frac{U(y)}{|x-y|^{n+1+s}}\,d{\mathcal{H}}^n_y\\&\qquad\qquad\qquad
=-\int_{{\partial^*E_\infty\cap\{x_n<-2\}}}\frac{\sigma_1}{|x-y|^{n+1+s}}\,d{\mathcal{H}}^n_y<
0,\end{split}\end{equation*}thanks to~\eqref{xpjwcq0gibtjm9uhgvc87trds3decd}.

Plugging this information into~\eqref{vdeOJSN.DJLM0w3oekfvf-023e-1k}
we obtain~\eqref{vdeOJSN.DJLM0w3oekfvf-023e-1} with~$\ell:=a$.

We also remark that, if~$V:=\nu_{\infty,n+1}$, we have that, in~$\Sigma_\infty\cap B_{\rho_0/2}(p_\infty)\cap\{x_n>0\}$,
\begin{equation}\label{vdeOJSN.DJLM0w3oekfvf-023e-2}
{\mathcal{L}}V\ge \ell V-C_0,\end{equation}
for some~$C_0>0$.
Indeed, using Lemma~\ref{CUTLE} and the $C^{1,\frac{1+s}2}$-regularity of~$\Sigma_\infty
\cap B_{2\rho_0}(p_\infty)$, we see that, for all~$x\in\Sigma_\infty\cap B_{\rho_0/2}(p_\infty)\cap\{x_n>0\}$,
$$ {\mathcal{L}}V=a(x)V(x)+\widetilde f(x)-V(x)\int_{{\Sigma_\infty\cap B_{\rho_0}(p_\infty)}}\frac{1-\nu_\infty(x)\cdot\nu_\infty(y)}{|x-y|^{n+1+s}}\,d{\mathcal{H}}^n_y,$$
with a bounded function~$\tilde f$ and
$$\left|V(x)\int_{{\Sigma_\infty\cap B_{\rho_0}(p_\infty)}}\frac{1-\nu_\infty(x)\cdot\nu_\infty(y)}{|x-y|^{n+1+s}}\,d{\mathcal{H}}^n_y\right|\le C_1|V(x)|\le C_1,$$
thus proving \eqref{vdeOJSN.DJLM0w3oekfvf-023e-2}.

We also remark that
\begin{equation}\label{193f84fnIHSknqwoefguj04twpyhl382irougjbmg}
{\mbox{$U>0$ in~$\Sigma_\infty\cap B_{\rho_0/2}(p_\infty)\cap\{x_n>0\}$.}}
\end{equation}
Indeed, suppose by contradiction that~$U(z_0)=0$ for some~$z_0\in\Sigma_\infty\cap B_{\rho_0/2}(p_\infty)\cap\{x_n>0\}$. Then, by~\eqref{vdeOJSN.DJLM0w3oekfvf-023e-1},
$$0\le
\int_{\Sigma_\infty\cap B_{\rho_0}(p_\infty)}\frac{ U(y)}{|z_0-y|^{n+1+s}}\,d{\mathcal{H}}^n_y=
{\mathcal{L}}U(z_0)< \ell U(z_0)=0,$$
which is absurd, thus establishing~\eqref{193f84fnIHSknqwoefguj04twpyhl382irougjbmg}.

Owing to~\eqref{vdeOJSN.DJLM0w3oekfvf-023e-1}, \eqref{vdeOJSN.DJLM0w3oekfvf-023e-2}, and~\eqref{193f84fnIHSknqwoefguj04twpyhl382irougjbmg},
we can therefore employ the boundary Harnack inequality in Lemma~\ref{LE:31}
(here with~$u:=U$ and~$v:=V$) and conclude that, in a neighborhood of~$p_\infty$,
$$\nu_{\infty,n+1}=V\le CU=C \sigma_1\nu_{\infty,n+1}-C\nu_{\infty,1},$$
for some~$C>0$.

This gives that, in this neighborhood,~$\frac{\nu_{\infty,1}}{\nu_{\infty,n+1}}\le \sigma_1-\frac{1}{C}$,
but this is in contradiction with~\eqref{cd0iksqp032tugjb230-ifwjv0vvtnh}.

The proof of~\eqref{X2.2ISJOLPXpaP} is thereby complete.
\end{proof}

\begin{appendix}

\section{{F}rom the ratio of normal components to the regularity of the boundary trace}

For completeness, we present here an auxiliary result, used in the proof of Theorem~\ref{THM1}
to obtain the regularity of the trace
of the nonlocal minimal surface~$\Sigma$ along~$\partial{\mathcal{C}}$.

\begin{lemma}\label{L2sw2I2C32rA}
Let~${\mathcal{S}}$ be a hypersurface of class~$C^1$ with exterior normal~$\nu$
and suppose that~${\mathcal{S}}\cap\overline{\mathcal{C}}$ corresponds to the graph~$\{x_{n+1}=u(x')$,
$x'\in\Omega\}$, with~$u:\Omega\to\R$ uniformly continuous.

Let~$\Gamma:=\{(x',u(x'))$, $x'\in\partial\Omega\}$.

Suppose that~$p\in\Gamma$ and that~$\frac{\tau\cdot\nu}{\nu_{n+1}}\in L^q( {\mathcal{S}}\cap{\mathcal{C}}\cap B_\rho(p))$, for every vector field~$\tau$ of class~$C^{1,1}$ that is tangent
to~$\partial{\mathcal{C}}$,
for
some~$q\in[1,+\infty]$ and~$\rho>0$.

Then, $\Gamma\cap B_\rho(p)$ is of class~$W^{1,q}$.

Similarly, if~$\frac{\tau\cdot\nu}{\nu_{n+1}}\in C^{k,\alpha}( {\mathcal{S}}\cap{\mathcal{C}}\cap B_\rho(p))$, for every~$\tau$ as above and some~$k\in\{0,1\}$, $\alpha\in(0,1]$ and~$\rho>0$, then~$\Gamma\cap B_\rho(p)$ is of class~$C^{k+1,\alpha}$.
\end{lemma}

\begin{proof} Up to a covering in local charts, we can suppose that~$\Omega\cap B_\rho(p)$ can be written as~$\Phi(U\times(0,1))$, for some domain~$U\subseteq\R^{n-1}$
and some injection~$\Phi:\R^{n}\to\R^n$ of class~$C^{2,1}$, with~$\Phi(U\times\{0\})=\partial\Omega$.

For every~$x\in\partial{\mathcal{C}}$, we know that~$x'\in\partial\Omega$.
Hence, for all~$j\in\{1,\dots,n\}$, we define~$\tau(x):=(\partial_j\Phi (\underline x),0)$,
where~$\underline x:=\Phi^{-1}(x')$,
and we stress that~$\tau$ is tangent
to~$\partial{\mathcal{C}}$.

Moreover, $\nu(x)=\frac{(-\nabla u(x'),1)}{\sqrt{|\nabla u(x')|^2+1}}$, for all~$x\in
{\mathcal{S}}\cap{\mathcal{C}}$ (but some care is needed when~$x\in
{\mathcal{S}}\cap\partial{\mathcal{C}}$, since the gradient of~$u$
may blow up there).

This gives that, for all~$x\in
{\mathcal{S}}\cap{\mathcal{C}}\cap B_\rho(p)$,
$$ -\frac{\tau(x)\cdot\nu(x)}{\nu_{n+1}(x)}
=\nabla u(x')\cdot\partial_j\Phi (\underline x)=\partial_j
(u\circ\Phi) (\underline x).$$
The identity above gives that the (local) regularity of~$u\circ\Phi$
is one derivative better than the (local)
regularity of~$\frac{\tau\cdot\nu}{\nu_{n+1}}$.
Since~$(\Phi,u\circ\Phi)\big|_{U\times\{0\}}$ is a parameterization for~$\Gamma$, the desired result follows.
\end{proof}

\end{appendix}

\vfill

\end{document}